\documentclass[12pt]{article}

\usepackage[russian,english]{babel}

\usepackage{graphicx}
\usepackage{amsmath}
\usepackage{amssymb}
\usepackage{amsthm}
\usepackage{geometry}
\usepackage{tikz}
\usepackage{enumitem}
\usepackage{csquotes}

\newcommand{\X}{\mathfrak X}
\newcommand{\Y}{\mathcal Y}
\newcommand{\Lo}{\mathsf{L}}

\newcommand{\eps}{\varepsilon}
\newcommand{\st}{\mathop{st}}

\newcommand*\circled[1]{\tikz[baseline=(char.base)]{
            \node[shape=circle,draw,inner sep=0.8pt] (char) {#1};}}
\newcommand{\rc}{\circled{R}}
\newcommand{\bco}{\circled{$b_1$}}
\newcommand{\bct}{\circled{$b_2$}}
\newcommand{\bci}{\circled{$b_i$}}
\newcommand{\mco}{\circled{$m_1$}}
\newcommand{\mct}{\circled{$m_2$}}
\newcommand{\mci}{\circled{$m_i$}}

\newcommand{\NN}{\mathbb{N}}

\newcommand{\QQ}{\mathbb{Q}}

\newcommand{\RR}{\mathbb{R}}

\newcommand{\set}[1]{\left\{#1\right\}}

\newcommand{\setdef}[2][x]{\set{#1\,\left |\,#2 \right .}}

\newcommand{\pair}[1]{ \langle #1 \rangle}

\newcommand{\romb}{\diamondsuit}
\newcommand{\Log}{{Log}}

\newcommand{\MLog}{{\mathcal{ML}}}
\newcommand{\logic}[1]{\mathsf{#1}}
\newcommand{\pmor}{\twoheadrightarrow}

\newcommand{\Top}{\mathop{Top}}

\newtheorem{theorem}{\textsc{Theorem}}[section]
\newtheorem{corollary}[theorem]{\textsc{Corollary}}
\newtheorem{lemma}[theorem]{\textsc{Lemma}}
\newtheorem{proposition}[theorem]{\textsc{Proposition}}

\theoremstyle{remark}
\newtheorem{remark}[theorem]{\textmd{\textsc{Remark}}}
\newtheorem*{remark*}{\textmd{\textsc{Remark}}}

\theoremstyle{definition}
\newtheorem{definition}[theorem]{Definition}
\newtheorem*{definition*}{Definition}

\title{	Topological square of logic S4.1}
\author{A. Kashchenko, A. Kudinov\\
HSE University}
\date{}

\begin{document}
\maketitle
\begin{abstract}
	In this paper, we find the axiomatization for the topological square of S4.1. This is the first known topological square of a modal logic that differs from both the fusion and the Kripke product. We also prove the finite model property and decidability for this logic.
\end{abstract}
\section{Introduction}
	It is often necessary to study logics with more than one modality. This motivates the study of different ways of combining modal logics. In this paper, we discuss the topological product of modal logics, which is a natural generalization of the better-studied Kripke product.
	
	The study of Kripke products of modal logics began with V.B. Shehtman in 1978 \cite{shehtman_two_dimensional}. In \cite{gabbay_shehtman_products}, it is shown that $L_1 \ast L_2 + \Box_1 \Box_2 p \leftrightarrow \Box_2 \Box_1 p + \romb_1 \Box_2 p \to \Box_2 \romb_1 p \subseteq L_1 \times L_2$.  At the same time, the topological product of modal logics was defined only in 2005 in \cite{van_benthem_multimodal}. Historically, topological semantics predates Kripke semantics: the foundational work of McKinsey and Tarski already established a connection between modal logic and topology in 1944 \cite{mckinsey_tarski_algebra_topology}. A central result of McKinsey and Tarski is that the modal logic of any topological space is an extension of $\logic{S4}$ and that $\logic{S4}$ is exactly the logic of all topological spaces.
	
	It is worth noting that topological products are sensitive to more structural properties of spaces that are not reflected in their one-dimensional modal logics. For instance, Kremer showed (\cite{kremer_topological_product_s4_s5}) that although the rational line and the real line validate the same modal logic, i.e. $Log(\QQ) = Log(\RR) = \logic{S4}$, their topological squares are different: $Log(\QQ \times \QQ) \neq Log(\RR \times \RR)$.
	
	The world of topological products is more diverse than the world of Kripke products.
	There is an example of a topological product that is equal to the fusion of logics: $\logic{S4} \times_t \logic{S4} = \logic{S4} \ast \logic{S4} = Log(\QQ \times \QQ)$ (\cite{van_benthem_multimodal}). There are examples of topological products of logics that are equal to their Kripke products: $\logic{S5} \times_t \logic{S5} = \logic{S5} \times \logic{S5}$ (cf. \cite{kremer_topological_product_s4_s5}).
	It is also known that the topological product lies  between the fusion and the Kripke product: $L_1 \ast L_2 \subseteq L_1 \times_t L_2 \subseteq L_1 \times L_2$ (\cite{kremer_topological_product_s4_s5}).
	For example, there are topological products that are equal to their semiproduct: the product of $\logic{S4}$ and $\logic{S5}$ (cf. \cite{kremer_topological_product_s4_s5}).
	A nontrivial example of the topological product that does not fall into these three categories was found in 2024 in \cite{kudinov_topological_product_mckinsey}: it turned out that the product $\logic{S4.1}\times_t\logic{S4}$ is equal neither to the fusion, nor to the Kripke product, nor to the Kripke semiproduct. In this paper, we examine the topological square of $\logic{S4.1}$, which also turns out to be nontrivial.

\section{Basic Concepts and Definitions}

\begin{definition}
 Let $ \mathrm{PROP} $ be a countable set of \emph{propositional variables}. The language of \emph{modal formulas} with $N$ modalities (or \emph{$N$-modal formulas}), denoted $\MLog_N$, is the smallest set of words over the alphabet $\mathrm{PROP} \cup \{\bot, \to, \Box_1, \ldots, \Box_N, (, )\}$, satisfying the following conditions:

\begin{enumerate}
    \item $\mathrm{PROP} \subseteq \MLog_N$;
    
    \item $\bot \in \MLog_N$;
    
    \item If $A, B \in \MLog_N$, then $(A\to B)\in\MLog_N$; 
    
    \item If $A \in \MLog_N $ and $ i \in \{1,\ldots,N\}$, then $ \Box_iA\in\MLog_N$.
\end{enumerate}
    
\end{definition}

In Backus–Naur form, it can be written as:
	$$
	A ::= p\; |\;\bot \; | \; (A \to A) \; | \; \Box_i A,
	$$
	where $p\in \mathrm{PROP}$ is a propositional variable and $\Box_i$ is a modal operator ($ i= 1, \ldots, N$).

Other logical connectives are treated as derived. The \emph{dual modality} is defined as follows: $ \Diamond_i \varphi:= \neg \Box_i \neg \varphi$. If the language contains only one modality, the index \( i \) is omitted, and we write $\Box$ and $ \MLog $.

\begin{definition} \label{def:modallogic}
A ($N$-)modal logic is the set of N-modal formulas such that:
	\begin{itemize}
		\item it contains all classical propositional tautologies;
		\item it contains the normality axioms
        $$\Box_i (p \to q) \to (\Box_i p \to \Box_i q);$$ 
  
		\item it is closed under the following inference rules:
		\begin{enumerate}
			\item \emph{Substitution} (\textsc{Sub}): \( \dfrac{A}{A[p/B]} \), where $A[p/B]$ is the result of uniformly substituting every occurrence of the variable $p$ by the formula $B$; 
			\item \emph{Modus Ponens} (\textsc{MP}): \( \dfrac{A,\ A \to B}{B} \);
			\item \emph{Necessitation} (\textsc{Nec}): \( \dfrac{A}{\Box_i A} \).
		\end{enumerate}
	\end{itemize}
	The minimal $N$-modal logic is denoted $\logic{K}_N$, and $\logic{K}_1 = \logic{K}$.
\end{definition}


Let \( \logic{L} \) be a modal logic, and let \( \Gamma \) be a set of formulas. Then \( \logic{L} + \Gamma \) defines a minimal modal logic containing \( \logic{L} \cup \Gamma \). 
If \( \Gamma \) consists of a single formula \( A \) we write \( \logic{L} + A \) instead of \( \logic{L} + \{A\} \).


\begin{definition}
	Let $ X \ne \varnothing $. A \emph{topology} on a set $ X $ is a set $ T $ of subsets of $ X $ such that:
	\begin{itemize}
		\item $\varnothing, X \in  T$;
		\item if $ U_1, U_2 \in T $ then $ U_1 \cap U_2 \in T $;
		\item if $\displaystyle \set{U_\alpha}_{\alpha\in I} \subseteq T$ then $\displaystyle \bigcup_{\alpha \in I} U_\alpha \in T $.
	\end{itemize}
	A \emph{topological space} is a pair $ (X, T) $. The elements of $T$ are called \emph{open} sets. If $x \in U \subseteq X$, we say that $U$ is a \emph{neighborhood} of the point $x$. If also $U \in T$, we say that $U$ is an \emph{open neighborhood} of point $x$.

    A \emph{Bitopological space} is a tuple $(X, T_1, T_2)$, where $T_1$ and $T_2$ are topologies on $X$.
\end{definition}

\begin{definition}
Let $(X, T)$ be a topological space. A family $\beta$ of subsets of $X$ is called a base for the topology $T$ if
\[
T = \left\{ \bigcup_{G \in M} G \mid M \subseteq \beta \right\}.
\]
\end{definition}

\begin{theorem}[\cite{engelking_general_topology}]
Let $X \neq \varnothing$ and $ B \subseteq 2^X $. Then $B$ is a base for a topology on $X$ if and only if
\begin{itemize}
    \item $B$ is a cover of $X$,
    \item $\forall G_1, G_2 \in B \; \forall x \in G_1 \cap G_2\ \exists G \in B \; x \in G \subseteq G_1 \cap G_2$.
\end{itemize}
\end{theorem}


\begin{definition}
	Let $ \X = (X, T_1, \ldots, T_N) $ be a space with $N$ topologies (in this paper $N\in\set{1,2}$). A \emph{valuation} on $ \X $ is a function
    \[
        V:\mathrm{PROP} \to 2^X.
    \]

    A pair $M = (\X, V) $ is called a \emph{topological model}.
\end{definition}

\begin{definition}
    The truth of a formula ``$\models$'' in a model $M = (\X, V)$ is defined by induction:
	\begin{align*}
		M, x \models p &\iff x \in V(p),\ \hbox{for $p\in \mathrm{PROP}$};\\
		M, x \not\models \bot; &\\
		M, x \models A \to B &\iff \bigl(M, x \models A \Rightarrow M, x \models B \bigr);\\
		M, x \models \Box_i A &\iff \exists U \in T_i \left ( x \in U \; \& \; \forall y \in U (M, y \models A) \right ).
	\end{align*}

    A formula $ A $ is \emph{true} in a topological model $ M $ (denoted $M \models A$) if $ \forall x\in \X (M,x \models A)$.

    A formula $ A $ is \emph{valid} in a space $ \X $ (denoted $\X \models A$) if $ \forall V ((\X,V) \models A)$.

    A formula $ A $ is \emph{valid} in a class of topological spaces $ \mathcal{C} $ (denoted $\mathcal{C} \models A$) if $ \forall \X \in \mathcal{C} (\X \models A)$.
    
    Let $\logic{L}$ be a logic. Space $\X$ is called \emph{$\logic{L}$-space} if $\X\models \logic{L}$.
\end{definition}

\begin{definition}
The logic of a class of topological spaces $\mathcal{C}$ is
\[
\Log(\mathcal{C}) = \{ A \mid \forall \X \in \mathcal{C} \, (\X \models A) \}.
\]
\end{definition}
	For a singleton, we omit the curly brackets (for instance $Log(\X) = Log(\set{\X})$).



\begin{theorem}[\cite{mckinsey_tarski_algebra_topology}]
For any topological space $\X$,
\begin{itemize}
	\item $\X \models \Box p \to p$ (reflexivity axiom);
	\item $\X \models \Box p \to \Box \Box p$ (transitivity axiom).
\end{itemize}
\end{theorem}

\begin{definition}
	We define logic
	\[ 
	\logic{S4} = \logic{K} + \set{\Box p \to p, \Box p \to \Box \Box p}.
	 \]
\end{definition}

\begin{theorem}[\cite{mckinsey_tarski_algebra_topology}]
    $\logic{S4}$ is the logic of the class of all topological spaces.
\end{theorem}

\vspace{20pt}

We also need Kripke semantics.

\begin{definition}
A \textit{Kripke $N$-frame} is a tuple  
\[
F = (W, R_1, \ldots, R_N),
\]
where \( W \neq \varnothing \) is a non-empty set, its elements are called \textit{possible worlds} or \textit{points}, and \( R_i \subseteq W \times W \) is a binary relation on $W$ for \( i \in \set{1, \ldots, N} \).
\end{definition}

\begin{definition}
    Let \( F = (W, R_1, \ldots, R_N) \) be a Kripke frame. A \emph{valuation} on $F$ is a function $V: \mathrm{PROP} \to 2^W$. Pair \( M = (F, V) \) is called a \textit{Kripke model}. The truth of a formula is defined by induction, similarly to the topological semantics. The only difference in the definition is:
	\begin{align*}
		M, x \models \Box_i A &\iff \forall y (xR_i y \Rightarrow M,y \models A).
	\end{align*}

    A formula $ A $ is \emph{true} in a model $ M $ (denoted $M \models A$) if $ \forall x\in W\; (M,x \models A)$.
	
    A formula $ A $ is \emph{valid} in a frame $ F $ (denoted $F \models A$) if $ \forall V \;((F,V) \models A)$.

    A formula $ A $ is \emph{valid} in a class of frames $ \mathcal{C} $ (denoted $\mathcal{C} \models A$) if \\ $ \forall F \in \mathcal{C}\; (F \models A)$.
    
    Let $\logic{L}$ be a logic. Frame $F$ is called \emph{$\logic{L}$-frame} if $F\models \logic{L}$.
\end{definition}

\begin{definition}
The \emph{logic of a class of Kripke frames $\mathcal{C}$} is defined as
\[
\Log(\mathcal{C}) = \{ A \mid \forall F \in \mathcal{C} \, (F \models A) \}.
\]
\end{definition}

For a singleton $\mathcal{C}$, we omit the curly brackets (for instance, we write $Log(F)$).

The soundness theorem for Kripke frames can be found in \cite{chagrov_zakharyaschev}.

A logic is \emph{Kripke complete} if it is the logic of some class of Kripke frames.
\bigskip

    For any reflexive and transitive Kripke frame $F=(W,R)$, we can define a topological space $\X = (W, T_R)$, 
    where $T_R$ is the topology with the base $B_R = \{R(x) \mid x \in W\}$. The set $B_R$ is, in fact, a base, since $\displaystyle W = \bigcup_{x \in W} R(x)$ and, for $z \in R(x) \cap R(y)$, we have $R(z) \subseteq R(x) \cap R(y)$ due to transitivity of $R$. We denote the resulting topological space by $\Top(F)$. Similarly, for a reflexive and transitive Kripke frame with two relations $F = (W, R_1, R_2)$ we denote $\Top_2(F) = (W, T_{R_1}, T_{R_2})$.

The following theorem is well known (cf. \cite{gabelaia_modal_def_2001})
\begin{theorem}
     For a reflexive and transitive Kripke frame $F=(W,R)$,  $Log(F) = Log(Top(F))$.
\end{theorem}

\begin{corollary}
	For a reflexive and transitive Kripke frame $F = (W, R_1, R_2)$, $Log(F) = Log(Top_2(F))$.
\end{corollary}

\begin{remark}
    Every point in the space $ \Top(F) $ has its minimal neighborhood. Such topologies are called \emph{Alexandrov topologies}. For an Alexandrov topology $T$ on $W$, there exists a unique reflexive and transitive Kripke frame such that $R(x)$ is the minimal neighborhood of $x$. Thus, there is a one-to-one correspondence 
    between reflexive and transitive Kripke frames and spaces with Alexandrov topology.
\end{remark}

A crucial tool for proving completeness results is a p-morphism:

\begin{definition}
	Let $F = (W, R_1, \ldots, R_N)$ and $G = (U,S_1, \ldots, S_N)$ be Kripke frames. A function $f:W \to U$ is called \emph{p-morphism} if
	\begin{enumerate}
		\item $xR_iy \Rightarrow f(x) S_i f(y)$ (monotonicity);
		\item $f(x) S_i u \Rightarrow \exists y (f(y) = u \;\&\; xR_iy)$ (lifting);
		\item $f$ is a surjection.
	\end{enumerate}
	
	We denote it as $f: F \twoheadrightarrow G$ and we write $F \twoheadrightarrow G$ if there exists a p-morphism from $F$ to $G$.
\end{definition}

The following lemma is well-known (see \cite{chagrov_zakharyaschev}).
\begin{lemma}[p-morphism lemma]\label{thm:pmorphism}
	If $F \twoheadrightarrow G$, then $Log(F) \subseteq Log(G).$
\end{lemma}

\begin{definition}
	Let $\X$ and $\Y$ be two topological spaces. Then a function $f:\X \to \Y$ is called \emph{open} (\emph{continuous}), if the image (preimage) of every open set is open.
\end{definition}

The analogue of a p-morphism for topological spaces is a surjective, open, and continuous map. For such a map, an analogue of Lemma \ref{thm:pmorphism} holds (see, e.g., \cite{gabelaia_modal_def_2001}):
\begin{lemma}\label{thm:top_pmorphism}
	Let $\X$ and $\Y$ be two topological spaces and let $f: \X\to \Y$ be a surjective, open, and continuous mapping. Then $Log(\X) \subseteq Log(\Y).$
\end{lemma}

We also call such mappings \emph{p-morphisms}. This causes no confusion, since the following holds:
\begin{lemma}
	For any two Kripke frames $ F_i = (W_i, R_i) $, where $ R_i $ is reflexive and transitive ($ i = 1,2 $), a surjective function $ f:W_1 \to W_2 $ is a p-morphism if and only if $ f $ is continuous and open with respect to the topologies $ T_{R_1}  $ and $ T_{R_2}$.
\end{lemma}

\begin{proof}
    \begin{enumerate}
        \item Let $f:F_1 \twoheadrightarrow F_2$. 
        It is sufficient to check openness and continuity on the open sets from a base of the topology.
        
        \textbf{(Continuity)}. We need to show that $f^{-1}(R_2(f(x)))$ is open in $T_{R_1}$. Take \hbox{$y\in f^{-1}(R_2(f(x)))$}. Then $f(y) \in R_2(f(x))$. Since $R_2$ is transitive, this implies 
        $$
        R_2(f(y)) \subseteq R_2(f(x)).
        $$ 
        By monotonicity of $f$, we have 
        $$
        f(R_1(y)) \subseteq R_2 (f(y)).
        $$
        Combining these gives
        \[ 
        f(R_1(y)) \subseteq R_2(f(x)),
        \]
        which is exactly equivalent to
        $$
        R_1(y) \subseteq f^{-1}(R_2 (f(x))).
        $$ 
        

        \textbf{(Openness)} Now we prove that $f(R_1(x))$ is open. Let $w \in f(R_1(x))$, i.e. $w = f(y)$ for some $y\in R_1(x)$. Let us prove that $R_2(w) \subseteq f(R_1(x))$. For $t \in R_2(w)$, by the lifting property there exists $z \in R_1(y)\; f(z) = t$. Then $t \in f(R_1(x))$. Hence, $f$ is open. 
    
        \item Let $f:W_1\to W_2$ be an open and continuous surjection. 
        
        Take any point $x \in W_1$ and its minimal open neighborhood $R_1(x)$. By openness of $f$ we have that $f(R_1(x))$ is open and contain the minimal open neighborhood of $f(x)$, i.e.
        \[ 
        R_2(f(x)) \subseteq f(R_1(x)).
        \]
        This yields the lifting property for $f$.
  
        
        Now we take the minimal neighborhood $R_2(f(x))$ for any $x\in W_1$. Since $f$ is continuous $f^{-1}(R_2(f(x)))$ is open and
        \[ 
        R_1(x) \subseteq f^{-1}(R_2(f(x))).
         \] 
        It follows that
        \[ 
        f(R_1(x)) \subseteq R_2(f(x)).
        \] 
        This yields the monotonicity for $f$.

    \end{enumerate}
\end{proof}

There are several ways of combining modal logics. The simplest way is the fusion: 

\begin{definition}
	Let $\logic{L_1}$ and $\logic{L_2}$ be two modal logics with one modality (unimodal logics). Then a \emph{fusion} of these logics is
	\[
	\logic{L_1} \ast \logic{L_2} = \logic{K_2} + \logic{L'_1} + \logic{L'_2};
	\] 
	where $\logic{L'_i}$ is the set of all formulas obtained  from $\logic{L_i}$ by replacing $\Box$ with $\Box_i$.
\end{definition}

The other is the product:

\begin{definition}(\cite{gabbay_shehtman_products})
    Let $F_i = (W_i, R_i)$ ($i = 1, 2$) be two Kripke frames. We define their \emph{product} as the frame with two relations $F_1 \times F_2 = (W_1 \times W_2, R_1', R_2')$, where
    
    \[
    (x, y)R_1' (z, t) \iff xR_1 z \ \&\ y = t,
    \]
    \[
    (x, y)R_2' (z, t) \iff x = z \ \&\ yR_2 t.
    \]
\end{definition}

\begin{definition}
    A \emph{product of two modal logics}  $ L_1 $ and $ L_2 $ is the following logic with two modalities:
\begin{align*}
	L_1 \times L_2 = Log(\setdef[F_1 \times F_2]{ F_1, F_2\hbox{ are Kripke frames,}\ F_1 \models L_1, F_2 \models L_2}).
\end{align*}
\end{definition}

\begin{theorem}[\cite{gabbay_shehtman_products}]
    For any two logics $\logic{L_1}$ and $\logic{L_2}$, we have
    \[
    \logic{L_1} \ast \logic{L_2} + com_{12} + com_{21} + chr \subseteq \logic{L_1} \times \logic{L_2},
    \]
    where
    \begin{align*}
        com_{12} = \Box_1\Box_2 p \to \Box_2\Box_1 p, \\
        com_{21} = \Box_2\Box_1 p \to \Box_1\Box_2 p, \\
        chr = \romb_1 \Box_2 p \to \Box_2 \romb_1 p.
    \end{align*}

\end{theorem}

\begin{remark}
    The inclusion may be strict. An example can be found in \cite{gabbay_shehtman_products}.
\end{remark}

For the logic $\logic{S4}$, we have the following.

\begin{theorem}[\cite{gabbay_shehtman_products}]
    \begin{align*}
        \logic{S4} \ast \logic{S4} + com_{12} + com_{21} + chr = \logic{S4} \times \logic{S4}.
    \end{align*}
\end{theorem}

The notion of product for modal logics depends on semantics. Thus, a topological product of modal logics can be defined.

\begin{definition}(\cite{van_benthem_multimodal})
	Let $ \X_1 = (X_1, T_1)$ and $ \X_2 = (X_2, T_2) $ be two topological spaces. We define their \emph{bitopological product} as the bitopological space  $ \X_1 \times \X_2 = (X_1 \times X_2, T_1^h, T_2^v )$ where the topology $ T_1^h $ has a base $ \setdef[U\times \set{x_2}]{U \in T_1\;\&\; x_2\in X_2} $ and the topology $ T_2^v $ has a base $ \setdef[\set{x_1}\times U]{x_1\in X_1 \;\&\;  U \in T_2} $. The topology $ T_1^h$ is called \emph{horizontal} and the topology $T_2^v$ is called \emph{vertical}.
\end{definition}

Such a product of topological spaces is similar to the product of Kripke frames. In fact this notion generalizes the Kripke product for reflexive and transitive frames:

\begin{lemma}
For any two Kripke reflexive and transitive frames $F_1$ and $F_2$, we have 
    $\Log(\Top(F_1) \times \Top(F_2)) = \Log(\Top_2(F_1 \times F_2)) =\Log(F_1 \times F_2)$.
\end{lemma}

\begin{definition}
    A \emph{topological product of modal logics}  $ L_1 $ and $ L_2 $ is the following logic with two modalities:
\begin{align*}
	&\logic{L_1} \times_t \logic{L_2} = Log(\setdef[\X_1 \times \X_2]{ \X_1, \X_2\hbox{ are topological spaces,}\ \X_1 \models \logic{L_1}, \X_2 \models \logic{L_2}}).
\end{align*}
\end{definition}

The topological product of modal logics does not generally coincide with the usual product, because there exist topological spaces in which points do not have minimal neighborhoods. However, it is known that it lies between the fusion and the usual product:

\begin{theorem}[\cite{kremer_topological_product_s4_s5}]
    Let $\logic{L_1}$ and $\logic{L_2}$ be two $\logic{S4}$-logics. Then
    \[
    \logic{L_1} \ast \logic{L_2} \subseteq \logic{L_1} \times_t \logic{L_2} \subseteq \logic{L_1} \times \logic{L_2}.
    \]
\end{theorem}

\section{McKinsey Axiom}

The formula $ A1 = \Box \romb p \to \romb \Box p $ is known as the \emph{McKinsey axiom}. The logic $\logic{S4.1}$ is defined as follows: 
$$
\logic{S4.1} = \logic{S4} + A1.
$$

\begin{lemma}[\cite{chagrov_zakharyaschev}]
The following formulas are equivalent over $\logic{S4}$:
\[
\Box \Diamond p \to \Diamond \Box p, \quad
\Diamond(\Diamond p \to \Box p), \quad
\Diamond(\Box p \vee \Box \neg p).
\]
That is, for any formula $A$, we have
\[
\logic{S4} + (\Box \Diamond A \to \Diamond \Box A)
=
\logic{S4} + \Diamond(\Diamond A \to \Box A)
=
\logic{S4} + \Diamond(\Box A \vee \Box \neg A).
\]
\end{lemma}

\begin{lemma}
    In a reflexive frame $F$, the satisfiability of the formula $\romb p \to \Box p$ (or $\Box p \vee \Box \neg p$) at a point $x$ under any valuation is equivalent to the maximality of that point (i.e. $R(x)=\{x\}$).
\end{lemma}

\begin{proof}
If $R(x)=\{x\}$, then:
\[
F,x \models \Diamond p \iff F,x \models p, \quad
F,x \models \Box p \iff F,x \models p.
\]
Hence $\Diamond p \to \Box p$ is equivalent to $p \to p$, so it holds under any valuation.

\medskip

Assume $R(x) \neq \{x\}$. Due to reflexivity there exists $y \in R(x)$ such that $y \neq x$.

Define a valuation $V$ by:
\[
V(p) = \{y\}.
\]
Then $F,x \models \Diamond p$ (since $y \in R(x)$ and $F,y \models p$) and $F,x \not\models \Box p$ (since $F,x \not\models p$). Thus $F,x \not\models \Diamond p \to \Box p$.
\end{proof}

\begin{lemma}[\cite{chagrov_zakharyaschev}]\label{lem:McKinseyProperty}
	In a $\logic{S4}$-frame $ F = (W,R) $, the validity of the formula $A1$ is equivalent to the following first-order condition:
	\[ 
	\forall w \in W \exists u \in W (w R u \land R(u) = \set{u} ),
	\]
	where $ R(u) = \setdef[t]{uRt} $.

    Informally speaking, this condition means that ``from every point, some maximal point is accessible''.
\end{lemma}

It is, however, interesting that in general in the non-reflexive and non-transitive case this formula is not elementary.

\begin{theorem}[\cite{vanbenthem_modal_formulae_relational_properties} and \cite{goldblatt_first_order_definability}]
    There is no first-order formula $\phi$ such that 
    \[ 
	\forall F \quad F \models_{FO} \phi \iff F \models A1,
	\]
    where symbol $\models_{FO}$ means truth in the sense of first-order logic.
\end{theorem}

In topological spaces, the analogue of maximality is isolation.

\begin{definition}
	In a topological space, a point $ x $ is called \emph{isolated} if a singleton set $ \set{x} $ is open. 
\end{definition}

\begin{definition}
    A topological space $ \X $ is called \emph{weakly scattered} if the set of all isolated points is dense in $\X$, that is, every nonempty open subset of $\X$ contains an isolated point.
\end{definition}

\begin{theorem}[\cite{gabelaia_modal_def_2001}]
	$ \logic{S4.1} $ is the logic of weakly scattered spaces.
\end{theorem}

\begin{lemma}\label{lem:isolated}
    If $x$ is an isolated point in a topological space $\X$, then the formula $\romb p \to \Box p$ is always valid in $x$.
\end{lemma}

\begin{lemma}\label{lem:1}
	Let $ \X_1 $ and $ \X_2 $ be two topological spaces, and suppose $ \X_1 $ is weakly scattered. Then
	\[ 
	\X_1 \times \X_2 \models \romb_1\Box_2(\romb_1 p \to \Box_1 p).
	\]
\end{lemma}

\begin{proof}
	Consider an arbitrary point $\langle x,y\rangle \in \X_1 \times \X_2$ 
	and an open neighborhood of the form $U \times \{y\}$, where $U \in T_1$ 
	and $\X_1=(X_1, T_1)$. Since $\X_1$ is weakly scattered, 
	the set $U$ contains a point $x'$ that is isolated in $\mathcal{X}_1$. 
	
	Then, for any $y' \in \X_2$, the point $\langle x',y'\rangle$ 
	is isolated in the horizontal topology, and hence, by Lemma~\ref{lem:isolated}, 
	\[ 
	\langle x',y'\rangle \models \Diamond_1 p \to \Box_1 p.
	\]
	It follows that
	\[ 
	\langle x',y\rangle \models \Box_2(\Diamond_1 p \to \Box_1 p).
	\]
	Since any neighborhood of $\langle x,y\rangle$ 
	contains a point where $\Box_2(\Diamond_1 p \to \Box_1 p)$ holds, 
	we conclude that
	\[ 
	\langle x,y\rangle \models \Diamond_1\Box_2(\Diamond_1 p \to \Box_1 p).
	\]
\end{proof}

For convenience, we use the following notation:
$$ 
\Lo = \logic{S4.1} \ast \logic{S4.1} + \romb_1\Box_2(\romb_1 p \to \Box_1 p) + \romb_2\Box_1(\romb_2 p \to \Box_2 p).
$$

\begin{corollary}\label{cor:S41xS41correctness}
$\logic{S4.1} \times_t \logic{S4.1} \supseteq \Lo$.
\end{corollary}

\begin{lemma}\label{lem:MKDuo_prop}
	For any $ \logic{S4.1} \ast \logic{S4} $-Kripke frame $ F = (W, R_1, R_2) $, we have
	\[ 
	F \models \romb_1\Box_2(\romb_1 p \to \Box_1 p) \iff \forall x \exists y (x R_1 y \ \& \ \forall z (y R_2 z \Rightarrow R_1 (z) = \set{z})).
	\]

     For a given $x$, we denote by $\mu(x)$ the point $y$ provided by this lemma.
\end{lemma}
    \begin{proof}
        Clearly, if the right-hand condition holds for a frame, then the formula $\Diamond_1 \Box_2 (\Diamond_1 p \to \Box_1 p)$ is valid in that frame. Therefore, it suffices to prove the left-to-right implication.
        
        Suppose that 
        \[
        \exists x \, \forall y \,(x R_1 y \implies \exists z \, (y R_2 z \ \wedge \ R_1(z) \neq \{z\})).
        \] 
        By reflexivity, $R_1(z) \neq \{z\}$ means that $\exists w \, (w \neq z \ \wedge \ z R_1 w)$. Let $$
        Z = \setdef[z]{\exists y (x R_1 y \land y R_2 z \land R_1(z) \ne \set{z})}\hbox{ and }S = R_1(Z).
        $$ 
        
        Assume that we can find a set $A \subseteq S$ such that 
        \[
        \forall z \in Z \ \exists u \in A \ \exists v \in S \setminus A ( z R_1 u \ \wedge \ z R_1 v).
        \] 
        Then, taking the valuation $V(p) = A$, we obtain that for all $y \in R_1(x)$ there exists $z \in R_2(y)$ such that there exist $u \in V(p)$ and $v \notin V(p)$ with $z R_1 u$ and $z R_1 v$. Hence, at $z$ the formula $\Diamond_1 p \to \Box_1 p$ is false, and therefore at $y$ the formula $\Box_2(\Diamond_1 p \to \Box_1 p)$ is false. Consequently, at $x$ the formula $\Diamond_1 \Box_2(\Diamond_1 p \to \Box_1 p)$ is false.

        It remains to show that such a set $A$ can be found. We proceed by transfinite induction. First, we choose a well-ordering $<$ on the set $S$ (which exists by Zermelo's theorem). For any point $\alpha\in S$ we want to prove the following: For every $\beta \leq \alpha$ there exist three sets $A_\beta, B_\beta, C_\beta \subseteq S$ satisfying the following conditions:
        \begin{enumerate}
        	\item $\forall \beta_1,\beta_2\leq\alpha\quad(\beta_1 \leq \beta_2 \implies A_{\beta_1} \subseteq A_{\beta_2} \ \&\ B_{\beta_1} \subseteq B_{\beta_2}\ \&\ C_{\beta_1} \subseteq C_{\beta_2})$;
        	\item $\forall \beta\leq\alpha\quad A_\beta \sqcup B_\beta = C_\beta$;
        	\item $\forall \beta\leq\alpha\quad \{\gamma \in S \ |\ \gamma \leq \beta\} \subseteq C_\beta$;
        	\item $\forall \beta\leq\alpha\quad \forall \delta \in C_\beta \cap Z\; \exists u\in A_\beta \exists v\in B_\beta\; (\delta R_1 u \ \&\ \delta R_1 v)$.
        \end{enumerate}
        
        We can then take $\displaystyle A = \bigcup_{\alpha \in S}A_\alpha$.
        
        Suppose it is true for every $\alpha' < \alpha$ and prove it for $\alpha$. We denote $\displaystyle\hat{A}_\alpha = \bigcup_{\alpha' < \alpha}A_{\alpha'}$, $\displaystyle\hat{B}_\alpha = \bigcup_{\alpha' < \alpha}B_{\alpha'}$, $\displaystyle\hat{C}_\alpha = \bigcup_{\alpha' < \alpha}C_{\alpha'}$. If $\alpha \in \hat{C}_\alpha$, then we can take $A_\alpha = \hat{A}_\alpha$, $B_\alpha = \hat{B}_\alpha$, $C_\alpha = \hat{C}_\alpha$. If $\alpha \notin \hat{C}_\alpha$ and $\alpha \notin Z$, then we take $A_\alpha = \hat{A}_\alpha$, $B_\alpha = \hat{B}_\alpha \cup \{\alpha\}$, $C_\alpha = \hat{C}_\alpha \cup \{\alpha\}$.
        
        Now suppose that $\alpha \notin \hat{C}_\alpha$ and $\alpha \in Z$. We consider the following cases: \begin{enumerate}
        	\item $\exists \beta \in R_1(\alpha) \cap Z \cap \hat{C}_\alpha$. Then $A_\alpha = \hat{A}_\alpha$, $B_\alpha = \hat{B}_\alpha \cup \{\alpha\}, C_\alpha = \hat{C}_\alpha \cup \{\alpha\}$.
        	
        	\item $R_1(\alpha) \cap \hat{C}_\alpha = \emptyset$. We construct a (finite or countably infinite) sequence $w_0, w_1, w_2, \ldots$ of points of $R_1(\alpha)$; for uniformity, set $w_0 = \alpha$. Suppose $w_0, \ldots, w_{N-1}$ have already been constructed ($N \ge 1$). If $w_{N-1} \notin Z$, the sequence ends with $w_{N-1}$. Otherwise, since $w_{N-1} \in Z$, there exists $w \neq w_{N-1}$ with $w_{N-1} R_1 w$; fix such a $w$. If $w \in \{w_0, \ldots, w_{N-1}\}$, the sequence ends with $w_{N-1}$; otherwise put $w_N = w$ and continue.
        	
        	By construction the points $w_0, w_1, w_2, \ldots$ are pairwise distinct, and by transitivity of $R_1$ they all belong to $R_1(\alpha) \subseteq S$; in particular, none of them lies in $\hat{C}_\alpha$. We take $A_\alpha = \hat{A}_\alpha \cup \{w_1, w_3, w_5, \ldots\}$, $B_\alpha = \hat{B}_\alpha \cup \{\alpha, w_2, w_4, \ldots\}$, $C_\alpha = \hat{C}_\alpha \cup \{\alpha, w_1, w_2, w_3, \ldots\}$.
        	
        	\item $\exists \beta \i===+==n R_1(\alpha) \cap \hat{C}_\alpha$ and $R_1(\alpha) \cap \hat{C}_\alpha \cap Z = \emptyset$. We take $C_\alpha = \hat{C}_\alpha \cup \{\alpha\}$. If $\beta \in \hat{A}_\alpha$, then we define $A_\alpha = \hat{A}_\alpha$, $B_\alpha = \hat{B}_\alpha \cup \{\alpha\}$. If $\beta \in \hat{B}_\alpha$, then we take $A_\alpha = \hat{A}_\alpha \cup \{\alpha\}$, $B_\alpha = \hat{B}_\alpha$.
        \end{enumerate}
        
        For each point $\alpha$ we constructed sets $A_\alpha, B_\alpha, C_\alpha$ that satisfy conditions 1--4. Setting $\displaystyle A = \bigcup_{\alpha \in S}A_\alpha$, $\displaystyle B = \bigcup_{\alpha \in S}B_\alpha$, and $\displaystyle C = \bigcup_{\alpha \in S}C_\alpha$, we have $\displaystyle A \sqcup B = C = S$, and moreover, 
        $$\forall \delta\in Z\; \exists u \in A \;\exists v \in B \; (\delta R_1 u \;\&\; \delta R_1 v).$$
    \end{proof}

\begin{lemma}\label{lem:MKDuo_prop2}
	For any $ \logic{S4.1} \ast \logic{S4.1} $-Kripke frame $ F = (W, R_1, R_2) $,
	\begin{align*}
    & F \models \romb_1\Box_2(\romb_1 p \to \Box_1 p) \text{ and } F \models \romb_2\Box_1(\romb_2 p \to \Box_2 p)\\
    &\implies \forall x \exists y \exists z\bigl(x R_1 y\;\&\; y R_2 z \;\&\; R_1(z) = \{z\} \;\&\; R_2(z) = \{z\}\bigr).
\end{align*}
\end{lemma}

\begin{proof}
    Note that, by Lemma \ref{lem:MKDuo_prop} it is also true that $\forall x \exists y (x R_2 y \ \& \ \forall z (y R_1 z \Rightarrow R_2 (z) = \set{z}))$. For every $x$, we denote such $y$ as $\nu(x)$.
    
    Consider a fixed $x$. Take $\mu(x)$ as $y$ from Lemma \ref{lem:MKDuo_prop}. Take $\nu(y) = \nu(\mu(x))$ as $z$. Then $x R_1 y$ and $y R_2 z$.

    Moreover, $R_1(z) = \set{z}$, since $z \in R_2(\mu(x))$. Also, $R_2(z) = \set{z}$, since $z \in R_1(\nu(y))$.
\end{proof}

\section{Canonicity of logic $\Lo$}

We use the standard construction of canonical model (see \cite{chagrov_zakharyaschev,blackburn_modal_logic}).

\begin{definition}
The \textit{canonical model} for a bimodal logic \( \mathcal{L} \subseteq \MLog_2 \) is a Kripke model of the following form:
\[
M_\mathcal{L} = (W_\mathcal{L}, R^{(1)}_\mathcal{L}, R^{(2)}_\mathcal{L}, V_\mathcal{L}),
\]
where:
\begin{itemize}
    \item \( W_\mathcal{L} \) is the set of all maximally consistent \( \mathcal{L} \)-theories,
    \item \( x R^{(i)}_\mathcal{L} y \iff \forall A (\Box_i A \in x \Rightarrow A \in y) \) for \( i \in \{1, 2\} \),
    \item \( V_\mathcal{L}(p) = \{ x \in W_\mathcal{L} \mid p \in x \} \) for every \( p \in \mathrm{PROP} \).
\end{itemize}
\end{definition}

\begin{theorem}[Canonical model theorem]\label{canon_model_prop}
For a logic \( \mathcal{L} \) and its canonical model \( M_\mathcal{L} = (W_\mathcal{L}, R_\mathcal{L}, V_\mathcal{L}) \) the following hold:
\begin{enumerate}
    \item \( M_\mathcal{L}, x \models A \iff A \in x; \)
    \item \( M_\mathcal{L} \models A \iff A \in \mathcal{L}. \)
\end{enumerate}
\end{theorem}

\begin{definition}
A logic \( \mathcal{L} \) is called \textit{canonical}, if it is valid in its own canonical frame: \( F_\mathcal{L} \models \mathcal{L} \).  
That is, every formula from \( \mathcal{L} \) is valid in its corresponding canonical model.
\end{definition}

\begin{theorem}[\cite{chagrov_zakharyaschev}]
If the logic is canonical, it is Kripke complete.
\end{theorem}

Let us show that logic \( \Lo \) is canonical.

\begin{lemma}\label{fusion_prop_1}
    For any formulas $\varphi_1, \hdots, \varphi_n$, the formula $\displaystyle\romb_1 \bigwedge_{i=1}^n\Box_2(\romb_1\varphi_i \to \Box_1\varphi_i) \in \logic{L}$ .
\end{lemma}

\begin{proof}
    Since $\romb_1\Box_2(\romb_1 p \to \Box_1 p) \in \logic{L}$, so $\romb_1\Box_2(\romb_1 \varphi \to \Box_1 \varphi) \in \logic{L}$ for any $\varphi$. If we take $\psi_i = \romb_1 \varphi_i \to \Box_1 \varphi_i$, ($i = 1,\ldots, n$), then $\romb_1\Box_2 \psi_1,\ldots,\romb_1\Box_2 \psi_n \in \logic{L}$.
    
    By rule $(\text{Nec})$, $\Box_1\romb_1\Box_2 \psi_1, \Box_1\romb_1\Box_2 \psi_2 \in \logic{L}$. And since $\logic{S4.1}$ contains the McKinsey axiom $\Box\romb p \to \romb\Box p$, it follows that $\romb_1\Box_1\Box_2 \psi_2 \in \logic{L}$. Using that $(\Box p \wedge \romb q) \to \romb(p \wedge q) \in \logic{S4}$, rules $(\text{Sub})$ and  $(\text{MP})$ we get $\romb_1(\romb_1\Box_2 \psi_1 \wedge \Box_1\Box_2 \psi_2) \in \logic{L}$. 
    
    Using the same line of reasoning we get $\romb_1\romb_1(\Box_2 \psi_1 \wedge \Box_2 \psi_2)$.
    
    Since $\logic{S4.1}$ contains formula $\romb\romb p \to \romb p$, it follows that $\romb_1(\Box_2 \psi_1 \wedge \Box_2 \psi_2) \in \logic{L}$. Similarly applying the same reasoning to $\romb_1(\Box_2 \psi_1 \wedge \Box_2 \psi_2)$ and $\romb_1\Box_2 \psi_3$ and so on, we obtain $\romb_1(\Box_2 \psi_1 \wedge \ldots \wedge \Box_2 \psi_n) \in \logic{L}$.

\end{proof}

\begin{lemma}\label{lem:canonisity_of_L}
    Let $F_L = (W_L, R_L^{(1)}, R_L^{(2)})$ be a canonical frame of logic $\logic{L}$. Then it satisfies the condition $\forall x \exists y (x R_L^{(1)} y \ \& \ \forall z (y R_L^{(2)} z \Rightarrow R_L^{(1)} (z) = \set{z}))$ and its mirror version.
\end{lemma}

\begin{proof}
    Let $\Gamma_0 \in W_L$. Consider
    \[
    \Gamma' = \set{\varphi \ |\ \Box_1\varphi\in\Gamma_0}\cup\set{\Box_2(\romb_1\varphi\to\Box_1\varphi)\ |\ \varphi\ \text{is a modal formula}}.
    \]
    We prove that $\Gamma'$ is $\widetilde{L}$-consistent. Suppose the contrary. Then
    \[
    \varphi, \Box_2(\romb_1\varphi_1\to\Box_1\varphi_1), \hdots, \Box_2(\romb_1\varphi_n\to\Box_1\varphi_n) \vartriangleright_{\widetilde{L}} \bot,
    \]
    for some $\Box_1\varphi\in\Gamma_0$ and formulas $\varphi_1, \hdots, \varphi_n$. 
    Here, a single formula $\varphi$ is sufficient, since if a derivation were to use several formulas $\psi_1, \ldots, \psi_m$ such that $\Box_1\psi_1\in\Gamma_0, \ldots, \Box_1\psi_m\in\Gamma_0$, then we could simply set $\displaystyle\varphi = \bigwedge_{i=1}^m \psi_i$. 
    The formula $\displaystyle \Box_1\bigwedge_{i=1}^m \psi_i$ then belongs to $\Gamma_0$ by the properties of the canonical model (Theorem \ref{canon_model_prop}). By the Deduction Theorem and the rule $( \text{Nec} )$, we have:
    \[
    \vartriangleright_{\widetilde{L}} \Box_1\varphi\to\Box_1\neg\bigwedge_{i=1}^n\Box_2(\romb_1\varphi_i\to\Box_1\varphi_i).
    \]
    Hence,
    \[
    \neg\romb_1\bigwedge_{i=1}^n\Box_2(\romb_1\varphi_i\to\Box_1\varphi_i) \in \Gamma_0.
    \]
    Contradiction with Lemma \ref{fusion_prop_1}.
 
    Now, we take a maximal $\widetilde{L}$-consistent extension of $\Gamma'$, which exists by the Lindenbaum Lemma.
    We will call it $\Gamma_1$. Clearly, $\Gamma_0 R_L^{(1)} \Gamma_1$. We show that $\Gamma_1$ is exactly the point $y$ we were looking for by contradiction. Suppose there exist $\Gamma_2, \Gamma_3 \in W_L \ :\ \Gamma_2 \neq \Gamma_3$, $\Gamma_1R_L^{(2)}\Gamma_2$ and $\Gamma_2R_L^{(1)}\Gamma_3$ (That is, from $\Gamma_1$ via $R_L^{(2)}$ we can reach a point that is not isolated with respect to $R_L^{(1)}$). Let $\varphi \in \Gamma_3 \setminus \Gamma_2$ (if instead $\varphi \in \Gamma_2 \setminus \Gamma_3$, replace $\varphi$ by $\neg\varphi$). Then, due to the reflexivity of $R_L^{(1)}$, we have $\romb_1\varphi\in\Gamma_2$ and $\Box_1\varphi\notin\Gamma_2$. But $\Box_2(\romb_1\varphi\to\Box_1\varphi)\in\Gamma_1$. Hence, $\romb_1\varphi\to\Box_1\varphi \in \Gamma_2$. This leads to contradiction.
    
    The proof of the mirror version of this property is analogous.
\end{proof}





\begin{theorem}
	The logic $ \Lo$ is canonical and, consequently, Kripke-complete.
\end{theorem}

\begin{proof}
    The logic $ \logic{S4.1} \ast \logic{S4.1} $ is canonical, since the axioms of reflexivity and transitivity are canonical (formula $A$ is canonical if $A \in L \implies F_L \models A$), and the McKinsey axiom is canonical in the presence of the transitivity axiom 
    (cf. \cite{chagrov_zakharyaschev}).
    
    The axioms $\romb_1\Box_2(\romb_1 p \to \Box_1 p)$ and $\romb_2\Box_1(\romb_2 p \to \Box_2 p$ are canonical due to Lemma \ref{lem:canonisity_of_L}.
\end{proof}


\section{Finite model property and decidability of $\Lo$}
\begin{definition}
   A logic has the \emph{finite model property} (FMP), if it is the logic of some class of finite frames. 
\end{definition}

\begin{theorem}(Harrop, \cite{chagrov_zakharyaschev})
	If a logic is finitely axiomatizable and has the FMP, it is decidable.
\end{theorem}

The logic $\logic{L}$ is obviously finitely axiomatizable. In the following we prove that it has the FMP.

\begin{definition}
    Let $\Psi$ be a set of formulas, closed under subformulas and let $M$ be a Kripke model. We define a relation on $M$:
    \begin{align*}
		x \equiv_{\Psi} y &\iff \forall A \in \Psi (M,x \models A \Leftrightarrow M,y \models A).
	\end{align*}
\end{definition}

\begin{proposition}
    The relation $\equiv_{\Psi}$ is an equivalence relation.
\end{proposition}

\begin{definition}
    Let $M = (W, R_1, R_2, V)$ be a Kripke model, $\Psi$ be a set of formulas closed under subformulas and $\equiv$ be an equivalence relation on $W$ such that $\equiv \subseteq \equiv_\Psi$. Then a model $M'=(W', R'_1, R'_2, V')$ is called \emph{filtration of $M$ through $\equiv$} if
    \begin{enumerate}
    \item $W' = W /_{\equiv}=\{[x] \; |\; x\in W\}$;
    
    \item $xR_iy \implies [x] R'_i [y]$ (for $i\in \set{1,2}$);

    \item if $\Box A \in \Psi$, $M,x \models \Box_i A$ and $[x]R'_i[y]$, then $M, y \models A$ (for $i\in \set{1,2}$);

    \item for any propositional variable $p \in \Psi$, we have $x\in V(p) \Leftrightarrow [x] \in V'(p)$.
    \end{enumerate}
\end{definition}

\begin{lemma} (Filtration lemma, see \cite{kudinov_shapirovsky})
    Let $M'$ be a filtration of $M$ through $\equiv$ such that $\equiv \subseteq \equiv_{\Psi}$, then
    \[
    \forall A \in \Psi \;\forall x \in W (M, x \models A \Leftrightarrow M', [x] \models A).
    \]
\end{lemma}

\begin{definition}
A filtration $M' = (W', R', V')$ is called the \emph{minimal filtration} if
\[
[x] R'_i [y] \quad \text{iff} \quad \exists x' \in [x] \, \exists y' \in [y] \; (x' R_i y') \quad\hbox{(for $i\in \set{1,2}$)}.
\]
\end{definition}

The transitivity property may not be preserved in a minimal filtration, so we need to take the transitive closure of the relation.
Let \( M' = (W', R'_1, R'_2, V') \) be the minimal filtration of a model \( M = (W, R_1, R_2, V) \) through \(\equiv\). We take \( M'' = (W', R_1'', R_2'', V') \), where \( R_i'' = (R_i')^* \) (a transitive closure).

\begin{lemma}
The model \( M'' \) is a filtration of \( M \) through \(\equiv\).
\end{lemma}
\begin{definition}
Model $M''$ is called the \emph{transitive filtration} of $M$ through $\equiv$.	
\end{definition}

\begin{theorem}
	The logic $ \Lo $ has FMP and decidable.
\end{theorem}

\begin{proof}
        Let $ A $ be a formula and let $ M = (F,V) $ be a model such that $ F\models \Lo $, $ M \not\models A $ and $ F = (W, R_1, R_2) $. 
	We introduce a partitioning of $ W $ into four subsets:
	$
	W = W_1 \sqcup W_2 \sqcup W_3 \sqcup W_4,
	$
	where
	\begin{itemize}
		\item $ W_1 $ is the set of all $ R_1 $-maximal, but not $R_2$-maximal points,
		\item $W_2$ is the set of all $R_2$-maximal, but not $R_1$-maximal points,
		\item $W_3$ is the set of all points, maximal with respect to both $R_1$ and $R_2$,
		\item $ W_4 = W \setminus (W_1 \sqcup W_2 \sqcup W_3)$.
	\end{itemize}  
	
	Let $\Psi$ be the set of all subformulas of formula $A$. 
	We define a relation $ \equiv $ as follows:
\[ 		
x \equiv y \iff x \equiv_\Psi y \ \hbox{and}\ \exists i \in \set{1, 2, 3, 4} (x,y \in W_i).
 \]	
	Let $ M' = (F', V')$ be the transitive filtration of model $ M $ through relation $ \equiv $. Reflexivity and transitivity are preserved in a transitive filtration.

    We show that in $F'$ for each $i \in\set{1,2}$, it holds that 
    $$
    \forall w \in W' \exists u \in W' (w R'_i u \land R'_i(u) = \set{u}).
    $$ 
    i.~e., the McKinsey axiom is valid (by Lemma \ref{lem:McKinseyProperty}). Let $[x] \in W'$. Since the original frame satisfies this property, there exists $y \in W$ such that $x R_i y \land R_i(y) = \set{y}$. By the definition of the filtration, we have $[x]R'[y]$. The transitive filtration is a transitive closure of the minimal filtration, and all the $z \in [y]$ are maximal with respect to $R_i$ by the definition of $\equiv$. Hence, $[y]$ is a maximal point with respect to $R'_i$.

    Similarly, one can show that in $F'$ we have $\forall x \exists y (x R_1 y \ \text{ and } \ \forall z (y R_2 z \Rightarrow R_1 (z) = \set{z}))$, i.~e., that $\romb_1\Box_2(\romb_1 p \to \Box_1 p)$ is valid (by Lemma \ref{lem:MKDuo_prop}). The same applies to the mirror image of this formula.
    
    Therefore, $ F' \models \Lo $. By the Filtration Lemma, we have $ M' \not\models A $, and $ M' $ is finite. Thus, FMP is established. 
\end{proof}

\section{Kripke frame for $\Lo$}

Let $ \mathbb{T}_{2,2} $ denote  the infinite transitive $(2,2)$-tree with two relations: $ \mathbb{T}_{2,2} = (T_{2,2}, R_1, R_2) $, where $T_{2,2} = \set{a_1, a_2, b_1, b_2}^*$ is the set of all finite words over the four-letter alphabet. For any two words $ \vec a, \vec b \in  T_{2,2}$,
\begin{align*}
	\vec a R_1 \vec b &\Leftrightarrow \hbox{there exists $ \vec c\in \set{a_1, a_2}^* $ such that} \ \vec b  = \vec a \cdot \vec c;\\
	\vec a R_2 \vec b &\Leftrightarrow \hbox{there exists $ \vec d\in \set{b_1, b_2}^*$ such that}\ \vec b = \vec a \cdot \vec d.
\end{align*} 
From now on, the dot “$\cdot$” denotes the concatenation of words.

Let $ \mathbb{\widetilde{T}} = (W, R_1', R_2')$ denote a Kripke frame, obtained from $ \mathbb{T}_{2,2}$ in two steps.

\textbf{Step 1.} ``Attaching'' to each point two copies of the infinite binary tree $T_2^{(1)} = (\set{a_1,a_2}^\ast, \sqsubseteq)$ and $T_2^{(2)} = (\set{b_1,b_2}^\ast, \sqsubseteq)$, where $\sqsubseteq$ denotes the ``prefix relation''. At every point the first copy $T_2^{(1)}$ is ``attached'' via the $ R'_2 $ relation, and inside it the points are accessible via the $R'_1$ relation. We denote by $ \eps^{(1)}$ the empty word, which serves as the root of this copy. Conversely, the second copy $T_2^{(2)}$ is ``attached'' via $ R'_1 $ relation, and inside it the points are accessible via $R'_2$ relation. We denote by $ \eps^{(2)}$ the empty word, which serves as the root of this copy.

\textbf{Step 2.} On every ``attached'' copy of tree $T_2^{(1)}$ to each of its points we ``attach'' via $R'_1$ a point which is maximal with respect to both relations.  On every ``attached'' copy of tree $T_2^{(2)}$ to each of its points we ``attach'' via $R'_2$ a point which maximal with respect to both relations.

Formally, the definition is as follows:
\begin{enumerate}

	\item $\Sigma = \set{a_1,a_2,b_1,b_2}$, $\Sigma_a = \set{a_1,a_2}$, $\Sigma_b = \set{b_1,b_2}$,
	
	\item $T_{2,2}$ is the (2,2)-tree over the alphabet $\Sigma$,
	
	\item $T_2^{(1)} =  (\Sigma_a^\ast, \sqsubseteq)$ is the infinite binary tree over the alphabet $\Sigma_a$,
	
	\item $T_2^{(2)} =  (\Sigma_b^\ast, \sqsubseteq)$ is the infinite binary tree over the alphabet $\Sigma_b$.

	\item\label{it:i1}  $ W = T_{2,2} \times \set{\rc} \sqcup T_{2,2} \times T_2^{(1)} \times \set{\bco, \mco} \sqcup  T_{2,2} \times T_2^{(2)} \times \set{\bct, \mct}$, where $\rc$ is a special new symbol representing a point in the main tree, $\bco$ and $\bct$ are special new symbols representing points in the attached trees $T_2^{(1)}$ and $T_2^{(2)}$ respectively (step 1), $\mco$ and $\mct$ are special new symbols representing attached maximal points (step 2).
	\item\label{it:i2} 
	$\forall \vec a, \vec a' \in T_{2,2}\quad \vec a R_1 \vec a' \Rightarrow \pair{\vec a, \rc} R'_1 \pair{\vec a', \rc}$;
	\item\label{it:i3} 
	$\forall \vec a \in T_{2,2}\quad \pair{\vec a, \rc} R'_1 \pair{\vec a, \eps^{(2)}, \bct} $;
        \item\label{it:i4} 
        $\forall \vec a \in T_{2,2}\;\forall \vec b,\vec b' \in T_2^{(1)}$ $\left(\vec b \sqsubseteq \vec b' \Rightarrow \pair{\vec a, \vec b, \bco} R'_1 \pair{\vec a, \vec b', \bco}\right)$
        \item\label{it:i5}
        $\forall \vec a \in T_{2,2}\;\forall \vec b \in T_2^{(1)}\quad\pair{\vec a, \vec b, \bco} R'_1 \pair{\vec a, \vec b, \mco}$
	\item $R'_1$ is the minimal reflexive and transitive relation satisfying items \ref{it:i2}, \ref{it:i3}, \ref{it:i4} and \ref{it:i5};
	\item\label{it:i6}
	$\forall \vec a, \vec a' \in T_{2,2}\quad \vec a R_2 \vec a' \Rightarrow \pair{\vec a, \rc} R'_2 \pair{\vec a', \rc}$;
        \item\label{it:i7} 
        $\forall \vec a \in T_{2,2}\quad \pair{\vec a, \rc} R'_2 \pair{\vec a, \eps^{(1)}, \bco} $;
	\item\label{it:i8} 
        $\forall \vec a \in T_{2,2}\;\forall \vec b,\vec b' \in T_2^{(2)}$ $\left(\vec b \sqsubseteq \vec b' \Rightarrow \pair{\vec a, \vec b, \bct} R'_2 \pair{\vec a, \vec b', \bct}\right)$
         \item\label{it:i9}
        $\forall \vec a \in T_{2,2}\;\forall \vec b \in T_2^{(2)}\quad\pair{\vec a, \vec b, \bct} R'_2 \pair{\vec a, \vec b, \mct}$
	\item  $R'_2$ is the minimal reflexive and transitive relation satisfying items \ref{it:i6}, \ref{it:i7}, \ref{it:i8} and \ref{it:i9}.
\end{enumerate}


On the following picture we depicted the bottom of frame $\mathbb{T}_{2,2}$ without the attached trees.

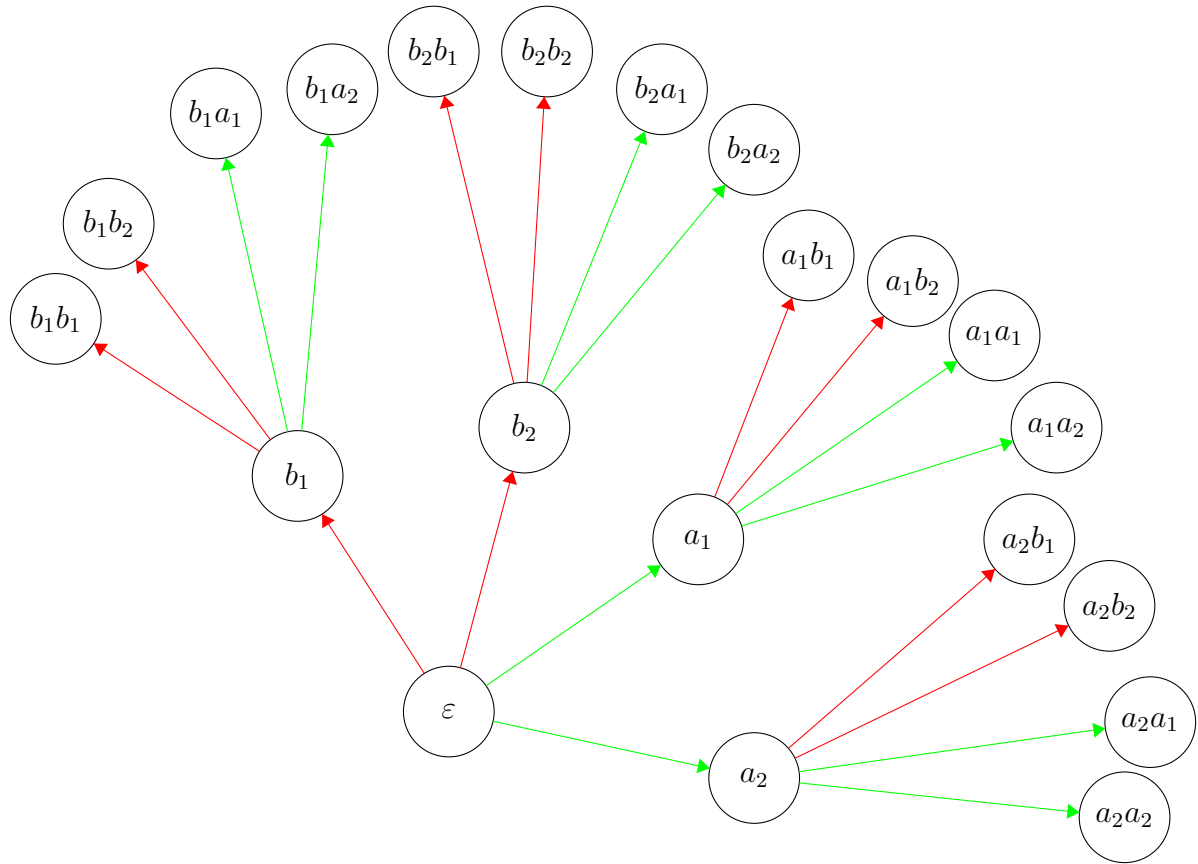
\begin{center}
\begin{tikzpicture}[scale=0.2]
\tikzstyle{every node}+=[inner sep=0pt]
\draw [black] (29.7,-48.8) circle (3);
\draw (29.7,-48.8) node {$\eps$};
\draw [black] (19.7,-33.2) circle (3);
\draw (19.7,-33.2) node {$b_1$};
\draw [black] (46.2,-37.4) circle (3);
\draw (46.2,-37.4) node {$a_1$};
\draw [black] (65.8,-23.9) circle (3);
\draw (65.8,-23.9) node {$a_1a_1$};
\draw [black] (69.9,-30) circle (3);
\draw (69.9,-30) node {$a_1a_2$};
\draw [black] (3.7,-22.8) circle (3);
\draw (3.7,-22.8) node {$b_1b_1$};
\draw [black] (7.2,-16.5) circle (3);
\draw (7.2,-16.5) node {$b_1b_2$};
\draw [black] (34.7,-30) circle (3);
\draw (34.7,-30) node {$b_2$};
\draw [black] (43.8,-7.6) circle (3);
\draw (43.8,-7.6) node {$b_2a_1$};
\draw [black] (49.9,-11.6) circle (3);
\draw (49.9,-11.6) node {$b_2a_2$};
\draw [black] (49.9,-53.2) circle (3);
\draw (49.9,-53.2) node {$a_2$};
\draw [black] (76.1,-49.5) circle (3);
\draw (76.1,-49.5) node {$a_2a_1$};
\draw [black] (74.4,-55.8) circle (3);
\draw (74.4,-55.8) node {$a_2a_2$};
\draw [black] (28.7,-5.1) circle (3);
\draw (28.7,-5.1) node {$b_2b_1$};
\draw [black] (36.2,-5.1) circle (3);
\draw (36.2,-5.1) node {$b_2b_2$};
\draw [black] (14.3,-9.2) circle (3);
\draw (14.3,-9.2) node {$b_1a_1$};
\draw [black] (22,-7.6) circle (3);
\draw (22,-7.6) node {$b_1a_2$};
\draw [black] (68.1,-37.4) circle (3);
\draw (68.1,-37.4) node {$a_2b_1$};
\draw [black] (73.4,-41.8) circle (3);
\draw (73.4,-41.8) node {$a_2b_2$};
\draw [black] (53.5,-18.6) circle (3);
\draw (53.5,-18.6) node {$a_1b_1$};
\draw [black] (60.4,-20.3) circle (3);
\draw (60.4,-20.3) node {$a_1b_2$};
\draw [red] (28.08,-46.27) -- (21.32,-35.73);
\fill [red] (21.32,-35.73) -- (21.33,-36.67) -- (22.17,-36.13);
\draw [green] (32.17,-47.09) -- (43.73,-39.11);
\fill [green] (43.73,-39.11) -- (42.79,-39.15) -- (43.36,-39.97);
\draw [red] (17.18,-31.57) -- (6.22,-24.43);
\fill [red] (6.22,-24.43) -- (6.61,-25.29) -- (7.16,-24.45);
\draw [red] (17.9,-30.8) -- (9,-18.9);
\fill [red] (9,-18.9) -- (9.08,-19.84) -- (9.88,-19.24);
\draw [green] (48.67,-35.7) -- (63.33,-25.6);
\fill [green] (63.33,-25.6) -- (62.39,-25.64) -- (62.95,-26.47);
\draw [green] (49.06,-36.51) -- (67.04,-30.89);
\fill [green] (67.04,-30.89) -- (66.12,-30.66) -- (66.42,-31.61);
\draw [red] (30.47,-45.9) -- (33.93,-32.9);
\fill [red] (33.93,-32.9) -- (33.24,-33.54) -- (34.21,-33.8);
\draw [green] (35.83,-27.22) -- (42.67,-10.38);
\fill [green] (42.67,-10.38) -- (41.91,-10.93) -- (42.83,-11.31);
\draw [green] (36.61,-27.69) -- (47.99,-13.91);
\fill [green] (47.99,-13.91) -- (47.09,-14.21) -- (47.87,-14.85);
\draw [green] (32.63,-49.44) -- (46.97,-52.56);
\fill [green] (46.97,-52.56) -- (46.29,-51.9) -- (46.08,-52.88);
\draw [green] (52.87,-52.78) -- (73.13,-49.92);
\fill [green] (73.13,-49.92) -- (72.27,-49.54) -- (72.41,-50.53);
\draw [green] (52.88,-53.52) -- (71.42,-55.48);
\fill [green] (71.42,-55.48) -- (70.67,-54.9) -- (70.57,-55.9);
\draw [red] (34,-27.08) -- (29.4,-8.02);
\fill [red] (29.4,-8.02) -- (29.1,-8.91) -- (30.08,-8.68);
\draw [red] (34.88,-27.01) -- (36.02,-8.09);
\fill [red] (36.02,-8.09) -- (35.47,-8.86) -- (36.47,-8.92);
\draw [green] (19.04,-30.27) -- (14.96,-12.13);
\fill [green] (14.96,-12.13) -- (14.65,-13.02) -- (15.62,-12.8);
\draw [green] (19.97,-30.21) -- (21.73,-10.59);
\fill [green] (21.73,-10.59) -- (21.16,-11.34) -- (22.16,-11.43);
\draw [red] (52.17,-51.23) -- (65.83,-39.37);
\fill [red] (65.83,-39.37) -- (64.9,-39.51) -- (65.56,-40.27);
\draw [red] (52.6,-51.89) -- (70.7,-43.11);
\fill [red] (70.7,-43.11) -- (69.76,-43.01) -- (70.2,-43.91);
\draw [red] (47.29,-34.6) -- (52.41,-21.4);
\fill [red] (52.41,-21.4) -- (51.66,-21.96) -- (52.59,-22.32);
\draw [red] (48.12,-35.09) -- (58.48,-22.61);
\fill [red] (58.48,-22.61) -- (57.59,-22.9) -- (58.36,-23.54);
\end{tikzpicture}
\end{center}

In the following picture we depicted a single point of $T_{2,2}$ in the frame $\mathbb{\widetilde{T}}_{2,2}$ looks with the attached trees and maximal points.

\begin{center}
\begin{tikzpicture}[scale=0.27, font=\tiny]
\tikzstyle{every node}+=[inner sep=0pt]
\draw [black] (41.6,-5.7) circle (3);
\draw (41.6,-5.7) node {$\pair{\vec a, \rc}$};
\draw [black] (35.2,-16.7) circle (3);
\draw (35.2,-16.7) node {$\pair{\vec a, \eps^{(1)}, \bco}$};
\draw [black] (48,-17.3) circle (3);
\draw (48,-17.3) node {$\pair{\vec a, \eps^{(2)}, \bct}$};
\draw [black] (22.8,-32.9) circle (3);
\draw (22.8,-32.9) node {$\pair{\vec a, a_1, \bco}$};
\draw [black] (38.2,-32.9) circle (3);
\draw (38.2,-32.9) node {$\pair{\vec a, a_2, \bco}$};
\draw [black] (45.5,-32.9) circle (3);
\draw (45.5,-32.9) node {$\pair{\vec a, b_2, \bct}$};
\draw [black] (59.2,-32.9) circle (3);
\draw (59.2,-32.9) node {$\pair{\vec a, b_1, \bct}$};
\draw [black] (32.6,-26) circle (3);
\draw (32.6,-26) node {$\pair{\vec a, a_2, \mco}$};
\draw [black] (15,-26) circle (3);
\draw (15,-26) node {$\pair{\vec a, a_1, \mco}$};
\draw [black] (50.2,-26) circle (3);
\draw (50.2,-26) node {$\pair{\vec a, b_2, \mct}$};
\draw [black] (65.4,-26) circle (3);
\draw (65.4,-26) node {$\pair{\vec a, b_1, \mct}$};
\draw [black] (25.8,-10.7) circle (3);
\draw (25.8,-10.7) node {$\pair{\vec a, \eps^{(1)}, \mco}$};
\draw [black] (56.7,-11.4) circle (3);
\draw (56.7,-11.4) node {$\pair{\vec a, \eps^{(2)}, \mct}$};
\draw [red] (40.09,-8.29) -- (36.71,-14.11);
\fill [red] (36.71,-14.11) -- (37.54,-13.67) -- (36.68,-13.16);
\draw [green] (43.05,-8.33) -- (46.55,-14.67);
\fill [green] (46.55,-14.67) -- (46.6,-13.73) -- (45.73,-14.21);
\draw [red] (47.53,-20.26) -- (45.97,-29.94);
\fill [red] (45.97,-29.94) -- (46.6,-29.23) -- (45.61,-29.07);
\draw [green] (35.75,-19.65) -- (37.65,-29.95);
\fill [green] (37.65,-29.95) -- (38,-29.07) -- (37.02,-29.25);
\draw [green] (33.38,-19.08) -- (24.62,-30.52);
\fill [green] (24.62,-30.52) -- (25.51,-30.19) -- (24.71,-29.58);
\draw [red] (49.75,-19.74) -- (57.45,-30.46);
\fill [red] (57.45,-30.46) -- (57.39,-29.52) -- (56.58,-30.1);
\draw [red] (61.21,-30.67) -- (63.39,-28.23);
\fill [red] (63.39,-28.23) -- (62.49,-28.49) -- (63.23,-29.16);
\draw [red] (47.19,-30.42) -- (48.51,-28.48);
\fill [red] (48.51,-28.48) -- (47.65,-28.86) -- (48.47,-29.42);
\draw [green] (36.31,-30.57) -- (34.49,-28.33);
\fill [green] (34.49,-28.33) -- (34.61,-29.27) -- (35.38,-28.64);
\draw [green] (20.55,-30.91) -- (17.25,-27.99);
\fill [green] (17.25,-27.99) -- (17.51,-28.89) -- (18.18,-28.14);
\draw [green] (32.67,-15.09) -- (28.33,-12.31);
\fill [green] (28.33,-12.31) -- (28.73,-13.17) -- (29.27,-12.32);
\draw [red] (50.48,-15.62) -- (54.22,-13.08);
\fill [red] (54.22,-13.08) -- (53.27,-13.12) -- (53.84,-13.95);
\draw [green] (21.88,-35.76) -- (20.12,-41.24);
\fill [green] (20.12,-41.24) -- (20.84,-40.64) -- (19.89,-40.33);
\draw [green] (23.17,-35.88) -- (23.83,-41.12);
\fill [green] (23.83,-41.12) -- (24.22,-40.27) -- (23.23,-40.39);
\draw [green] (37.42,-35.8) -- (35.98,-41.2);
\fill [green] (35.98,-41.2) -- (36.67,-40.56) -- (35.7,-40.3);
\draw [green] (38.2,-35.9) -- (38.2,-41.1);
\fill [green] (38.2,-41.1) -- (38.7,-40.3) -- (37.7,-40.3);
\draw [red] (45.18,-35.88) -- (44.62,-41.12);
\fill [red] (44.62,-41.12) -- (45.2,-40.37) -- (44.21,-40.27);
\draw [red] (46.15,-35.83) -- (47.35,-41.17);
\fill [red] (47.35,-41.17) -- (47.66,-40.28) -- (46.68,-40.5);
\draw [red] (58.69,-35.86) -- (57.91,-40.34);
\fill [red] (57.91,-40.34) -- (58.54,-39.64) -- (57.56,-39.47);
\draw [red] (59.68,-35.86) -- (60.42,-40.34);
\fill [red] (60.42,-40.34) -- (60.78,-39.47) -- (59.79,-39.63);
\end{tikzpicture}
\end{center}

\begin{lemma}\label{lem:tildaT}
    $ \Log(\mathbb{\widetilde{T}}) = \Lo$.
\end{lemma}

\begin{proof}
It is easy to see that $ \mathbb{\widetilde{T}} \models \Lo $. Indeed, the relations are transitive and reflexive by construction. From each point of the main tree $\mathbb{T}_{2,2}$, an $R_1$-maximal point in the attached tree is reachable. Moreover, each point accessible from this point via $R_2$ is also maximal with respect to $R_1$. From each point of the attached tree $T_2^{(1)}$, an attached maximal (with respect to both relations) point is reachable. And each point in the attached tree $T_2^{(2)}$ is maximal with respect to $R_1$. Hence, $\mathbb{\widetilde{T}} \models \logic{S4.1}\ast\logic{S4} + \romb_1\Box_2(\romb_1 p \to \Box_1 p)$. Similarly, one can show that $\mathbb{\widetilde{T}} \models \logic{S4}\ast\logic{S4.1} + \romb_2\Box_1(\romb_2 p \to \Box_2 p)$.

Consider now $A\notin \Lo$. The logic $ \Lo $ has FMP, so there is a rooted finite frame $ F $ such that $ F \models \Lo$ and $ F \not\models A $. Now it is enough to show that $ \mathbb{\widetilde{T}} \pmor F $.

In \cite{van_benthem_modal_space} a p-morphism $ f: \mathbb{T}_{2,2} \pmor F$ is described.

By Lemma \ref{lem:MKDuo_prop}, for each point $ w\in F $ there exists $ \mu^{(1)}(w) \in R_1(w) $ such that every $ u \in R_2(\mu^{(1)}(w)) $ is $R_1$-maximal. And also there exists $ \mu^{(2)}(w) \in R_2(w) $ such that every $ u \in R_1(\mu^{(2)}(w)) $ is $R_2$-maximal.

Clearly, for all $w \in F$, the points $\mu^{(2)}(\mu^{(1)}(w))$ and $\mu^{(1)}(\mu^{(2)}(w))$ are maximal with respect to both relations.

For each $ w $, we fix two p-morphisms: 

$ h^{(1)}_w: \mathbb{T}_2 \pmor \left (R_1(\mu^{(2)}(w)), R_1|_{R_1(\mu^{(2)}(w))}\right )$ and $ h^{(2)}_w: \mathbb{T}_2 \pmor \left (R_2(\mu^{(1)}(w)), R_2|_{R_2(\mu^{(1)}(w))}\right )$. 
 
 We define $ g: \mathbb{\widetilde{T}} \pmor F $:
	\begin{align*}
		g\left (\pair{\vec a, \rc}\right ) &= f(\vec a);\\
		\vec b \in T_2^{(1)}\implies g\left (\pair{\vec a, \vec b, \bco}\right ) &=h^{(1)}_{f(\vec a)}(\vec b)\\
            \vec b \in T_2^{(1)}\implies g\left (\pair{\vec a, \vec b, \mco}\right ) &=\mu^{(1)}(h^{(1)}_{f(\vec a)}(\vec b))\\
            \vec b \in T_2^{(2)}\implies g\left (\pair{\vec a, \vec b, \bct}\right ) &=h^{(2)}_{f(\vec a)}(\vec b)\\
            \vec b \in T_2^{(2)}\implies g\left (\pair{\vec a, \vec b, \mct}\right ) &=\mu^{(2)}(h^{(2)}_{f(\vec a)}(\vec b)).
	\end{align*}
	
	Let us check that  $ g $ is a p-morphism.
	
	The map $ g $ is surjective, since $ f $ is surjective.
	It is monotone, since maps $ f $, $ h^{(1)}_w$ and $h^{(2)}_w$ are all monotone, and all relations are transitive. Finally, the reader can easily check the lifting property by considering the three cases: $x = \pair{\vec a, \rc}$, $ x = \pair{\vec a, \vec b, \bci}$ and $x = \pair{\vec a, \vec b, \mci}$; and using the lifting property of $ f $, $ h^{(1)}_w$ and $h^{(2)}_w$.
\end{proof}

\section{Main result}

\begin{theorem}\label{thm:main}
	$\logic{S4.1} \times_t \logic{S4.1} = \Lo$.
\end{theorem}

By Corollary \ref{cor:S41xS41correctness} we have the left-to-right inclusion. For the right-to-left inclusion,   we need to construct a weakly scattered space $\X$ such that there exists an open and continuous surjective map  $ f: \X \times \X \to \Top_2(\mathbb{\widetilde{T}})$. 
Using Lemma \ref{lem:tildaT} and Theorem \ref{thm:top_pmorphism} we get $$
\logic{S4.1} \times_t \logic{S4.1} \subseteq \Log(\X \times \X) \subseteq \Log(\mathbb{\widetilde{T}}) = \Lo.
$$

The rest of this section is dedicated to constructing space $\X$ and map $f$.

\begin{definition}
	A \emph{path with stops} over $\mathbb{T}_2 = (\set{1,2}^*, \sqsubseteq) $ is a tuple (or a word) $x_1 \ldots x_n$, where $x_i \in \set{0,1,2}$. We define function $ f_F $ recursively on the set of all paths with stops:
	\begin{itemize}
		\item $ f_F(\eps) = \eps $;
		\item $ f_F(\vec a 0) = f_F(\vec a) $;
		\item $ f_F(\vec a 1) = f_F(\vec a)1 $;
		\item $ f_F(\vec a 2) = f_F(\vec a)2 $.
	\end{itemize}

	
\end{definition}

\begin{definition}
	A \emph{pseudo-infinite path (with stops)} over $\mathbb{T}_2$ is an infinite sequence of 0, 1 and 2, containing only finitely many non-zero values.
	We denote the infinite sequences of zeros as $0^\omega$.
	So we call $\alpha$ a pseudo-infinite path if $\alpha = \vec{a}0^\omega$ for some $\vec{a} \in \set{0,1,2}^*$.
	Let $W_\omega$ be the set of all pseudo-infinite paths over $\mathbb{T}_2$.
	We define functions $st: W_\omega \to \NN$ and $f_\omega:W_\omega \to T_2$. Let $\alpha = x_1 \ldots x_n \ldots \in W_\omega$. Then
	\begin{align*}
		\st(\alpha) &= \min \setdef[N]{\forall k> N (x_k = 0)};\\
		\alpha\lceil_k &= x_1 \ldots x_k;\\ 
		f_\omega(\alpha) &= f_F\left (\alpha\lceil_{\st(\alpha)}\right );\\
        U_k(\alpha) &= \setdef[\beta \in W_\omega]{\alpha \lceil_m = \beta \lceil_m, \text{ where }m=\max(k, st(\alpha))},\hbox{ for }\alpha \in W_\omega, k\in \NN.
	\end{align*}
\end{definition}

The following proposition is obvious.

\begin{proposition}
     For any $\alpha \in W_\omega$ and $s \in \NN$,
    \begin{enumerate}
		\item If $s - 1 < st(\alpha)$ then $U_{s-1}(\alpha) = U_s(\alpha)$;
		\item If $s - 1 \geq st(\alpha)$ then $U_{s-1}(\alpha) = U_s(\alpha) \sqcup U_s(\alpha\lceil_{s-1}10^{\omega}) \sqcup U_s(\alpha\lceil_{s-1}20^{\omega})$.
	\end{enumerate}
\end{proposition}

That is, the neighborhood $U_s(\alpha)$ with a smaller index $s$ consists of neighborhoods with larger indices.

\begin{proposition}\label{prop:top_1}
    For any $\alpha, \beta \in W_\omega$ and $0<k, m$, one and only one of the following holds: 
	\begin{enumerate}
		\item $U_k(\alpha) \cap U_m(\beta) = \varnothing$;
		\item $U_k(\alpha) \subseteq U_m(\beta)$ or $U_m(\beta) \subseteq U_k(\alpha)$.
	\end{enumerate}
\end{proposition}

The proof is straightforward.

\begin{lemma}
	The family $B = \setdef[U_{k}(\alpha)]{\alpha \in W_\omega, k>0}$ forms a base for a topology.
\end{lemma}
\begin{proof}
	It follows from the Proposition (\ref{prop:top_1}) that the intersection of any two elements of $B$ is either empty or coincides with one of them and therefore $B$ forms a base for a topology.
\end{proof}

The topology with this base we denote as $T_\omega$ and define $\Y = (W_\omega, T_\omega)$.

\begin{proposition}[cf. \cite{kudinov_neighborhood_frames}]
	Function $f_\omega: \Y \to \Top(T_2)$ is a p-morphism.
\end{proposition}

We define a function $g:W_\omega\times W_\omega \to T_{2,2}$. Let $\alpha = x_1x_2\ldots \in W_\omega$ and $\beta = y_1y_2\ldots \in W_\omega$. Recall that $T_{2,2} = \set{a_1, a_2, b_1, b_2}^*$. For convenience, we set $a_0 = b_0 = 0$.
\[  
g(\alpha, \beta) = f'_\omega(a_{x_1}b_{y_1}a_{x_2}b_{y_2}\ldots),
\]
where $f'_\omega$ is a function erasing zeros similar to $f_\omega$ but defined on the set of infinite sequences of symbols from the set $\set{0, a_1, a_2, b_1, b_2}$. Since $\alpha$ and $\beta$ end with tails of zeros, the image $g(\alpha, \beta)$ is an elements of $T_{2,2}$.

\begin{proposition}[cf. \cite{kudinov_neighborhood_frames}] \label{prop:pmorphismg}
	Function $g:\Y\times \Y \pmor Top_2(\mathbb{T}_{2,2})$ is a p-morphism.
\end{proposition}


We denote: $\NN_{\ge n} = \setdef[k\in \NN]{k\ge n}$. 

We define $\X = (X, T)$ as follows:
\begin{align*}
	X &= W_\omega \times \NN,\\	
	U'_k (\alpha,0) &= \left (U_k(\alpha) \times \set{0}\right ) \cup \left (U_k(\alpha) \times \NN_{\ge k}\right ), \\
	U'_k (\alpha,n) &= \set{\pair{\alpha,n}}, \hbox{ where } n\ge 1. 
\end{align*}
The sets $U'_k(\alpha, n)$ form a base for the topology $T$. It is enough to check that the sets of the form $U'_k(\alpha, n)$ either do not intersect or one is contained in the other.

Points of the form $\pair{\alpha, n}$  ($n\ge 1$) are isolated, so each base element contains isolated points. Therefore,
\begin{lemma}\label{lem:S41xS4corr}
	The space $\X$  is weakly scattered and $\X \models A1$.
\end{lemma}

Let us consider a bitopological space $\X \times \X$.

\begin{proposition}
In $\X \times \X$, for any $n, k > 0$,
\begin{enumerate}[label=\arabic*.] 
    \item $(\pair{\alpha, n}, \pair{\beta, 0})$ is an isolated point in the horizontal topology.
    \item $(\pair{\alpha, 0}, \pair{\beta, k})$ is an isolated point in the vertical topology
    \item $(\pair{\alpha, n}, \pair{\beta, k})$ is an isolated point in both topologies.
\end{enumerate}
\end{proposition}

\begin{lemma}
	$\X \times \X \pmor \Top_2(\mathbb{\widetilde{T}})$.
\end{lemma}
\begin{proof}
    To construct the required p-morphism we use the p-morphism from Proposition \ref{prop:pmorphismg}:
    $$
    g:\Y\times \Y \pmor Top_2(\mathbb{T}_{2,2}).
    $$

    We also define functions $f^{(1)}_{\omega}:W_{\omega}\to T^{(1)}_2$ and $f^{(2)}_{\omega}:W_{\omega}\to T^{(2)}_2$ as follows
    \begin{align*}
    f^{(1)}_{\omega}(\alpha) &= m_1(f_{\omega}(\alpha))\text{, where } m_1:\set{1,2}^\ast\to\set{a_1,a_2}^\ast\text{, }m_1\text{ respects concatenation}\\
    & \text{and }m_1(1) = a_1,\; m_1(2)=a_2,\\
    f^{(2)}_{\omega}(\alpha) &= m_2(f_{\omega}(\alpha))\text{, where } m_2:\set{1,2}^\ast\to\set{b_1,b_2}^\ast\text{, }m_2\text{ respects concatenation}.\\
    & \text{and }m_2(1) = b_1,\; m_2(2)=b_2\\
    \end{align*}

    These functions erase zeros and map infinite words from $W_\omega$ to trees $T^{(1)}_2$ and $T^{(2)}_2$ respectively.

    We now define a function $ f:\X \times \X \to \mathbb{\widetilde{T}} $. For $n,k >0$,
	\begin{align*}
		f(\pair{\alpha, 0}, \pair{\beta, 0}) &= \pair{g(\alpha, \beta), \rc},\\
		 f(\pair{\alpha, n}, \pair{\beta, 0})
		&=\pair {g(\alpha\lceil_n \cdot 0^\omega, \beta\lceil_n \cdot 0^\omega), f^{(2)}_{\omega}(\gamma), \bct},\hbox{ where } \beta = \beta\lceil_n\;\cdot\;\gamma,\\
            f(\pair{\alpha, 0}, \pair{\beta, k})
		&=\pair {g(\alpha\lceil_k \cdot 0^\omega, \beta\lceil_k \cdot 0^\omega), f^{(1)}_{\omega}(\gamma), \bco},\hbox{ where } \alpha = \alpha\lceil_k\;\cdot\;\gamma,\\
     f(\pair{\alpha, n}, \pair{\beta, k})
		&=\pair {g(\alpha\lceil_k \cdot 0^\omega, \beta\lceil_k \cdot 0^\omega), f^{(1)}_{\omega}(\gamma), \mco},\hbox{ where } \alpha = \alpha\lceil_k\;\cdot\;\gamma, \hbox{ and } n \ge k\\
            f(\pair{\alpha, n}, \pair{\beta, k})
		&=\pair {g(\alpha\lceil_n \cdot 0^\omega, \beta\lceil_n \cdot 0^\omega), f^{(2)}_{\omega}(\gamma), \mct},\hbox{ where } \beta = \beta\lceil_n\;\cdot\;\gamma, \hbox{ and } n < k.\\
	\end{align*}

\begin{lemma}
   Map $f$ is surjective.
\end{lemma}
\begin{proof}
     It follows from the surjectivity of $g$ and $f_\omega$.
\end{proof}
	
 Let us check that $f$ is open and continuous.\\

 (I) $T^h$-openness.\\

a) For any $n>0$, the point $(\pair{\alpha, n}, \pair{\beta, 0})$ is h-isolated. Its image is a point in the attached tree $T_2^{(2)}$, i.e. also an h-isolated point.

b) For any $n,k>0$, the point  $ (\pair{\alpha, n}, \pair{\beta, k})$ is h-isolated. Its image is one of the maximal points in the attached tree, i.e. also an h-isolated point.

c) For point $(\pair{\alpha, 0}, \pair{\beta, 0})$, let us consider its horizontal neighborhood $U'_s(\alpha, 0) \times \set{\pair{\beta, 0}}$. Without loss of generality, we can assume that $s\ge\max(\st(\alpha), \st(\beta))$. 

\begin{lemma} If $s\ge\max(\st(\alpha), \st(\beta))$, then
    $$f(U'_s(\alpha, 0) \times \set{\pair{\beta, 0}}) = g \left(U_s(\alpha), \beta \right)\times\set{\rc} \cup  g \left(U_s(\alpha), \beta \right) \times \set{\eps^{(2)}}\times \set{\bct}.$$
\end{lemma}
\begin{proof}
    The set $U'_s(\alpha, 0)$ consists of two parts, so:
    
    $$U'_s(\alpha, 0) \times \set{\pair{\beta, 0}} = \left(U_s(\alpha)\times\set{0} \times \set{\pair{\beta, 0}}\right) \cup \left(U_s(\alpha)\times\NN_{\ge s} \times \set{\pair{\beta, 0}}\right).$$ 
    Let us consider the image of each part separately.

    \begin{itemize}
        \item $f\left(U_s(\alpha)\times\set{0} \times \set{\pair{\beta, 0}}\right) = \pair{g \left(U_s(\alpha)\times \set{\beta}\right), \rc}$ by the definition of $f$.
        \item If $\alpha' \in U_s(\alpha)$ and $d \geq s$, then
        $$
        f\left(\pair{\alpha', d}, \pair{\beta, 0}\right) = \pair {g(\alpha'\lceil_d \cdot 0^\omega, \beta\lceil_d \cdot 0^\omega), f^{(2)}_{\omega}(\gamma), \bct},\hbox{ where } \beta = \beta\lceil_d\;\cdot\;\gamma.
        $$
    
        But $d \geq \st(\beta)$, and therefore $\beta\lceil_d\cdot 0^\omega = \beta, \; \gamma = 0^\omega$ and 
        $$
        f\left(\pair{\alpha', d}, \pair{\beta, 0}\right) = \pair {g(\alpha'\lceil_d \cdot 0^\omega, \beta), \eps^{(2)}, \bct}.
        $$

        Also $d \geq \st(\alpha)$, and therefore $\alpha'\lceil_d \cdot 0^\omega \in U_s(\alpha)$ and 
        $$
        f\left(U_s(\alpha)\times\NN_{\ge s} \times \set{\pair{\beta, 0}}\right) \subseteq \pair{g \left(U_s(\alpha), \beta\right), \eps^{(2)}, \bct}.
        $$

        \item Now we prove the inclusion in the opposite direction: let $\vec{w} \in g (U_s(\alpha), \beta)$. That is $g(\alpha', \beta) = \vec{w}$ for some $\alpha' \in U_s(\alpha)$. Then for $d = \max(st(\alpha'), s, \st(\beta))$ we have $\alpha'\lceil_d\cdot 0^\omega = \alpha'$ and 
        $\pair {g(\alpha'\lceil_d \cdot 0^\omega, \beta), \eps^{(2)}, \bct} = \pair {g(\alpha', \beta), \eps^{(2)}, \bct}$. Therefore, \hbox{$f\left(\pair{\alpha', d}, \pair{\beta, 0}\right) = \pair{\vec{w}, \eps^{(2)}, \bct}$}.

        So, we have $\alpha'\lceil_d \in U_s(\alpha)$ and $f\left(U_s(\alpha)\times\NN_{\ge s} \times \set{\pair{\beta, 0}}\right) \supseteq \pair{g \left(U_s(\alpha), \beta\right), \eps^{(2)}, \bct}$.
    \end{itemize}
\end{proof}

The set $g \left(U_s(\alpha), \beta\right)\times\set{\rc, \eps^{(2)}\times\bct}$ is h-open, since $g$ is open, and points of the form $\pair{\vec{w}, \eps^{(2)}, \bct}$ are maximal with respect to the first relation.

d) For point $(\pair{\alpha, 0}, \pair{\beta, k})$ (where $k>0$), let us consider its horizontal neighborhood $U'_s(\alpha, 0) \times \set{\pair{\beta, k}}$. Without loss of generality, we can assume that $s\ge\max(\st(\alpha), \st(\beta), k+1)$. 
\begin{lemma}
	If $s\ge\max(\st(\alpha), \st(\beta), k+1)$, then
    $$
    \displaystyle f(U'_s(\alpha, 0) \times \set{\pair{\beta, k}}) = \{\vec{w}\} \times f_{\omega}^{(1)}(V) \times \{\bco, \mco\}
    $$ 
    for some $\vec{w} \in T_{2,2}$ and open $V \subseteq W_\omega$.
\end{lemma}

\begin{proof}
    The set $U'_s(\alpha, 0)$ consists of two parts, so:    
    $$
    U'_s(\alpha, 0) \times \set{\pair{\beta, k}} = \left(U_s(\alpha)\times\set{0} \times \set{\pair{\beta, k}}\right) \cup \left(U_s(\alpha)\times\NN_{\ge s} \times \set{\pair{\beta, k}}\right).$$ 
    Let us consider the image of each part separately.

    \begin{itemize}
        \item The image of the first part is equal to 
        \begin{align*}
        f\left(U_s(\alpha)\times\set{0} \times \set{\pair{\beta, k}}\right) &= \bigcup_{\alpha' \in U_s(\alpha)}f(\pair{\alpha', 0}, \pair{\beta, k})\\
        &= \bigcup_{{\substack{\alpha'\in U_s(\alpha)\\ \alpha'=\alpha'\lceil_k\gamma}}}\pair {g(\alpha'\lceil_k \cdot 0^\omega, \beta\lceil_k \cdot 0^\omega), f^{(1)}_{\omega}(\gamma), \bco}.        	
        \end{align*}
        
        Since $s \geq k+1$, for $\alpha' \in U_s(\alpha)$ we have $\alpha'\lceil_k = \alpha\lceil_k$. If $\vec{w} = g(\alpha\lceil_k \cdot 0^\omega, \beta\lceil_k \cdot 0^\omega)$, then 
        $$ 
        f\left(U_s(\alpha)\times\set{0} \times \set{\pair{\beta, k}}\right) = \bigcup_{{\substack{\alpha'\in U_s(\alpha)\\ \alpha'=\alpha'\lceil_k\gamma}}}\pair {\vec{w}, f^{(1)}_{\omega}(\gamma), \bco}.
        $$

        \item Consider the image $f\left(U_s(\alpha)\times\set{d} \times \set{\pair{\beta, k}}\right)$ for $d \geq s \geq k+1$. From the definition of $f$, it is clear that the only difference from the previous case is the replacement of $\bco$ by $\mco$.
        \begin{align*}
        	f&\left(U_s(\alpha)\times\set{d} \times \set{\pair{\beta, k}}\right) = \bigcup_{\alpha' \in U_s(\alpha)}f(\pair{\alpha, d}, \pair{\beta, k})\\ 
        	&= \bigcup_{{\substack{\alpha'\in U_s(\alpha)\\ \alpha'=\alpha'\lceil_k\gamma}}}\pair {g(\alpha'\lceil_k \cdot 0^\omega, \beta\lceil_k \cdot 0^\omega), f^{(1)}_{\omega}(\gamma), \mco} = \bigcup_{{\substack{\alpha'\in U_s(\alpha)\\ \alpha'=\alpha'\lceil_k\gamma}}}\pair {\vec{w}, f^{(1)}_{\omega}(\gamma), \mco}
        \end{align*}
        where $\vec{w} = g(\alpha\lceil_k \cdot 0^\omega, \beta\lceil_k \cdot 0^\omega)$.
    \end{itemize}
\end{proof}

The set $\{\vec{w}\} \times f_{\omega}^{(1)}(V) \times \{\bco, \mco\}$ is h-open, since $f_\omega^{(1)}$ is an open map.
\bigskip

  (II) $T^h$-continuity.

a) For $\vec b \in T_2^{(2)}$, a point $\pair{\vec a, \vec b, \bct}$  is h-isolated and by the definition of $f$ its preimage consists of points of the form $(\pair{\alpha, n}, \pair{\beta, 0})$ for $n \geq 1$. And all these points are h-isolated.\\

b) For $\vec b \in T_2^{(2)}$, a point $\pair{\vec a, \vec b, \mct}$  is h-isolated and by the definition of $f$ its preimage consists of points of the form $(\pair{\alpha, n}, \pair{\beta, k})$ for $k > n \geq 1$. And all these points are h-isolated.

c) For $\vec b \in T_2^{(1)}$, consider a point $\pair{\vec a, \vec b, \bco}$ and its minimal neighborhood 
$$ 
R'_1(\pair{\vec a, \vec b, \bco}) = \left(\bigcup_{\vec c \in \set{a_1,a_2}^\ast}\set{\pair{\vec a, \vec b \cdot \vec c, \bco}}\right) \cup \left(\bigcup_{\vec c \in \set{a_1,a_2}^\ast}\set{\pair{\vec a, \vec b \cdot \vec c, \mco}}\right).
$$
Its preimage also consists of two parts: the preimage of the left bracket and the preimage of the right one.

For $\vec{c} \in \{a_1,a_2\}^*$, let us consider a preimage $f^{-1}(\pair{\vec a, \vec b \cdot \vec c, \bco})$. By the definition of $f$ it consists of points of the form $(\pair{\alpha, 0}, \pair{\beta, k})$ for $k\ge 1$. We take such point $(\pair{\alpha, 0}, \pair{\beta, k}) \in f^{-1}(\pair{\vec a, \vec b \cdot \vec c, \bco})$. By the definition $g(\alpha\lceil_k0^{\omega}, \beta\lceil_k0^{\omega}) = \vec a$, so
$$
U'_s(\pair{\alpha, 0})\times\pair{\beta, k} = \left (U_s(\alpha) \times \set{0}\right ) \cup \left (U_s(\alpha) \times \NN_{\ge s}\right ) \times \pair{\beta, k},
$$ 
for any $s \ge \max(k, \st(\alpha), \st(\beta))$. And for any $\alpha' \in U_s(\alpha)$,  $\alpha\lceil_k = \alpha'\lceil_k$, and therefore all points of neighborhood $U'_s(\pair{\alpha, 0})\times\pair{\beta, k}$ are mapped by $f$ onto the same attached tree. Since the first $\st(\alpha)$ symbols of $\alpha$ and $\alpha'$ coincide, their images lie in the subtree rooted at $\vec b \cdot \vec c$.

Now for some $\vec{c} \in \{a_1,a_2\}^*$ let us consider the preimage $f^{-1}(\pair{\vec a, \vec b \cdot \vec c, \mco})$. By the definition of $f$ it consists of points of the form $(\pair{\alpha, n}, \pair{\beta, k})$ for $n \ge k\ge 1$. These points are h-isolated. Clearly, image of its minimal neighborhood is $\set{\pair{\vec a, \vec b \cdot \vec c, \mco}}$.

d) For $\vec b \in T_2^{(1)}$, a point $\pair{\vec a, \vec b, \mco}$ is h-isolated and by the definition of $f$ its preimage consists of points of the form $(\pair{\alpha, n}, \pair{\beta, k})$ for $ n \ge k \ge 1$. And these points are h-isolated.

e) It remains to consider the point $\pair{\vec a, \rc}$. Its minimal h-neighborhood is equal to 
$$
\left(\bigcup_{\vec b \in \set{a_1,a_2}^\ast}\set{\pair{\vec a\cdot\vec b,\rc}}\right)\cup\left(\bigcup_{\vec b \in \set{a_1,a_2}^\ast}\set{\pair{\vec a\cdot\vec b,\eps^{(2)},\bct}}\right).$$ 
Its preimage also consists of two parts: the preimage of the left bracket and the preimage of the right one.

For some $\vec{b} \in \set{a_1,a_2}^*$, let us consider the preimage $f^{-1}(\pair{\vec a\cdot\vec b,\rc})$. 
By the definition of $f$ it consists of points of the form $(\pair{\alpha, 0}, \pair{\beta, 0})$. 
Let us take  $(\pair{\alpha, 0}, \pair{\beta, 0}) \in f^{-1}(\pair{\vec a\cdot\vec b,\rc})$. 
Then we have $g(\alpha, \beta) = \vec a\cdot\vec b$. 
The function $g$ is continuous, so $\exists s\left(  g(U_s(\alpha) \times \set{\beta}) \subseteq R_1(\vec{a}\cdot\vec{b})\right)$. 
Then for $m\ge \max(s, \st(\alpha), \st(\beta))$ let us take neighborhood 
\begin{align*}
U'_m(\pair{\alpha, 0})\times\pair{\beta, 0} &=(U_m(\alpha)\times\set{0}\cup U_m(\alpha)\times\NN_{\ge m})\times\pair{\beta, 0}\\
&=U_m(\alpha)\times\set{0} \times\pair{\beta, 0} \cup U_m(\alpha)\times\NN_{\ge m} \times\pair{\beta, 0}
\end{align*}
The image of its first part is $f(U_m(\alpha)\times\set{0}\times\pair{\beta, 0}) \subseteq R'_1(\pair{\vec{a}\cdot\vec{b}, \rc})$. Note also that, due to the conditions on $m$, for any $ \alpha' \in U_m(\alpha)$ and  $ n' \ge m$, we have  $\alpha'\lceil_{n'}0^{\omega}=\alpha'$ and $\beta\lceil_{n'}0^{\omega} = \beta$. Hence, $g(\alpha'\lceil_{n'}0^{\omega}, \beta\lceil_{n'}0^{\omega}) = g(\alpha', \beta) \in g(U_m(\alpha), \beta) \subseteq R_1(\vec{a}\cdot\vec{b})$. And for the same $\alpha'$ and $n'$, we have 
$$
f(\pair{\alpha', n'}, \pair{\beta, 0})=\pair{g(\alpha'\lceil_{n'}0^\omega, \beta\lceil_{n'}0^\omega), f_{\omega}^{(2)}(\gamma),\bct}=\pair{g(\alpha', \beta), \eps^{(2)},\bct},
$$
then
\[  
\pair{g(\alpha', \beta), \eps^{(2)},\bct} \in \pair{g(U_m(\alpha), \beta), \eps^{(2)},\bct} \subseteq R'_1(\pair{\vec a\cdot\vec b, \rc}).
\]

For some $\vec{b} \in \{a_1,a_2\}^*$, consider the preimage $f^{-1}(\pair{\vec a\cdot\vec b,\eps^{(2)},\bct})$. By the definition of $f$ it consists of points of the form $(\pair{\alpha, n}, \pair{\beta, 0})$ for $n\ge 1$. These points are h-isolated. The image of the minimal neighborhood of such point is $\set{\pair{\vec a\cdot\vec b,\eps^{(2)},\bct}}$.
\bigskip

(III) $T^v$-openness and $T^v$-continuity.

The verification for the vertical topology repeats the horizontal one with
the roles of the two coordinates interchanged and with the following
replacements: $R'_1 \leftrightarrow R'_2$, $\bco \leftrightarrow \bct$,
$\mco \leftrightarrow \mct$, $\eps^{(1)} \leftrightarrow \eps^{(2)}$,
$f^{(1)}_{\omega} \leftrightarrow f^{(2)}_{\omega}$,
$T_2^{(1)} \leftrightarrow T_2^{(2)}$. 

The only asymmetry of $f$ under this translation is on the diagonal: for
$n = k \ge 1$ the point $(\pair{\alpha, n}, \pair{\beta, k})$ is mapped to
an $\mco$-point rather than an $\mct$-point. This does not affect the
argument: every point with $n, k \ge 1$ is isolated in both topologies of
$\X \times \X$, and its image is a point of the form
$\pair{\,\cdot\,, \,\cdot\,, \mco}$ or $\pair{\,\cdot\,, \,\cdot\,, \mct}$,
which is maximal with respect to both relations, i.e.\ isolated in both
topologies of $\Top_2(\mathbb{\widetilde{T}})$. In all cases of the
translated argument such points occur only in the
``isolated~$\mapsto$~isolated'' steps, where the distinction between
$\mco$ and $\mct$ is irrelevant.

Thus, $f$ is a surjection which is open and continuous with respect to
both topologies, i.e.\ $f: \X \times \X \pmor \Top_2(\mathbb{\widetilde{T}})$.

\end{proof}


\section{Conclusion}
We have studied a particular example of a topological product. And yet we are left with lots of open problems. Among them: a description of topological products of other extensions of $\logic{S4}$, such as $\logic{S4.2}$, $\logic{S4.3}$, $\logic{S4.1.2}$ and others. Moreover, if instead of interpreting $\Diamond$ as the closure operator we interpret it as derivative operator, it becomes possible to study topological products of modal logics weaker than $\logic{S4}$. For example, $\logic{D4}$ or $\logic{K4}$. However, the most interesting problem is to find criteria, or at least sufficient conditions, under which the topological product of two logics coincides with their fusion or with the Kripke product.

\section{Acknowledgments}
The research leading to these results has received funding from the Basic Research Program at HSE University (HSE-BR-2025-84) for the second author.

Both authors are members of scientific group of the Junior Leader research grant 24-7-2-42-1 of the ``BASIS'' Foundation.


\bibliographystyle{plain}
\bibliography{bible}
\end{document}